\documentclass[smallextended,numbook,runningheads]{svjour3}
\smartqed
\usepackage{mathptmx}
\usepackage{array}
\usepackage{mathrsfs}
\usepackage{amsfonts}
\usepackage{boxedminipage}
\usepackage{amssymb,amsmath,color,bm}
\usepackage{pifont,eurosym}
\usepackage{epsfig,graphicx,epsf}
\usepackage{hyperref}
\usepackage{algorithm}
\usepackage{algorithmic}
\usepackage{multirow}
\usepackage{float}
\newcommand{\tf}{\textbf}

\journalname{BIT}
\begin{document}
\title{Robust Dropping Criteria for F-norm Minimization Based Sparse Approximate
Inverse Preconditioning\thanks{Supported
by National Basic Research Program of China 2011CB302400 and the
National Science Foundation of China (No. 11071140).}}
\author{Zhongxiao Jia \and
   Qian Zhang}
\institute{Zhongxiao Jia \at
Department of Mathematical Sciences, Tsinghua University, Beijing 100084, People's Republic of China.\\
\email{jiazx@tsinghua.edu.cn}.
\and
Qian Zhang\at
Department of Mathematical Sciences, Tsinghua University, Beijing 100084, People's Republic of China.\\
\email{qianzhang.thu@gmail.com}
}
\date{Received: date / Accepted: date}

\maketitle


\begin{abstract}
Dropping tolerance criteria play a central role in Sparse Approximate
Inverse preconditioning. Such criteria have received, however, little attention and
have been treated heuristically in the following manner: If the size of an entry is
below some empirically small positive quantity, then it is set to zero. The meaning
of "small" is vague and has not been considered rigorously. It has not been clear
how dropping tolerances affect the quality and effectiveness of a preconditioner $M$.
In this paper, we focus on the adaptive Power Sparse Approximate Inverse algorithm
and establish a mathematical theory on robust selection criteria for dropping
tolerances. Using the theory, we derive an adaptive dropping criterion that is used
to drop entries of small magnitude dynamically during the setup process of $M$.
The proposed criterion enables us to make $M$ both as sparse as possible as well as to be
of comparable quality to the potentially denser matrix which is obtained without
dropping. As a byproduct, the theory applies to static F-norm minimization based
preconditioning procedures, and a similar dropping criterion is given that can be
used to sparsify a matrix after it has been computed by a static sparse approximate
inverse procedure.
In contrast to the adaptive procedure, dropping in the static
procedure does not reduce the setup time of the matrix but makes the application
of the sparser
$M$ for Krylov iterations cheaper.
Numerical experiments reported
confirm the theory and illustrate the robustness and effectiveness of the dropping
criteria.

\keywords{Preconditioning\and sparse approximate inverse
\and dropping tolerance selection criteria
\and F-norm minimization\and adaptive\and static}

\subclass{65F10}
\end{abstract}

\section{Introduction}
Preconditioned Krylov subspace methods \cite{saad2003} are
among the most popular iterative solvers for large sparse linear
system of equations
\begin{equation*}
 Ax=b,
\end{equation*}
where $A$ is a nonsingular and nonsymmetric (non-Hermitian)
$n\times n$ matrix and $b$ is an $n$-dimensional vector.
Sparse approximate inverse (SAI) preconditioning aims to construct sparse
approximations of $A^{-1}$ directly and is nowadays one class of important
general-purpose preconditioning for Krylov solvers. There are two typical kinds of SAI
preconditioning approaches. One constructs a factorized sparse
approximate inverse (FSAI). An effective algorithm of this kind is the approximate
inverse (AINV) algorithm, which is derived from the incomplete (bi)conjugation
procedure \cite{benzi1996sparse,benzi1998sparse}.
The other is based on F-norm minimization and is inherently parallelizable.
It aims to construct $M\approx A^{-1}$
by minimizing $\|AM-I\|_F$ for a specified pattern
of $M$ that is either prescribed in advance or determined adaptively,
where $\|\cdot\|_F$ denotes the F-norm of a matrix. A hybrid version, i.e.,
the factorized approximate inverse (FSAI) preconditioning
based on F-norm minimization, has been introduced by Kolotilina
and Yeremin\cite{KOLOTILINA1993}.
FSAI is generalized to block form, called BFSAI
in \cite{janna2010block}. An adaptive algorithm in \cite{janna2011adaptive}
is presented that generates automatically the nonzero
pattern of the BFSAI preconditioner.
In addition, the idea of F-norm minimization is generalized in \cite{holland2005} by
introducing a sparse readily inverted {\em target} matrix $T$. $M$ is then computed
by minimizing $\|AM-T\|_{F,H}$ over a space of matrices with a prescribed sparsity
pattern, where $\|\cdot\|_{F,H}$ is the generalized F-norm defined
by $\|B\|^2_{F,H}=\langle B,B\rangle_{F,H}=trace(B^THB)$ with $H$ being some symmetric
(Hermitian) positive definite matrix, the superscript $T$ denotes the transpose of
a matrix or vector, and is replaced by the conjugate transpose for a complex
matrix $B$. A good comparison of factorized SAI and F-norm minimization
based SAI preconditioning approaches can be found in \cite{benzi1999comparative}.
SAIs have been shown to provide effective smoothers
for multigrid; see, e.g., \cite{Broker200261,broker:1396,sedlacek,Tang2000}.
For a comprehensive survey on preconditioning techniques, we refer the reader to
\cite{benzi2002preconditioning}.

In this paper, we focus on F-norm minimization based SAI preconditioning, where
a central issue is to determine the sparsity pattern of $M$ effectively.
There has been much work on a-priori pattern prescriptions, see, e.g.,
\cite{Benson1982,carpentieri2000,chow2000,huckle99,Tang1999}. Once the pattern of $M$
or its envelop is given, the computation of $M$ is straightforward by solving $n$
independent least squares (LS) problems and $M$ is then further sparsified generally.
This is called a static SAI preconditioning procedure. Huckle \cite{huckle99}
has compared different a-priori sparsity patterns and established effective upper
bounds for the sparsity pattern of $M$ obtained by the famous adaptive SPAI
algorithm \cite{grote1997parallel}. He shows that
the patterns of $(I+A)^k$, $(I +|A|+|A^T|)^k A^T$ and $(A^TA)^k A^T$ for small $k$
can be good envelop patterns of a good $M$. These patterns are very useful for reducing
communication times when distributing
and then computing $M$ in a distributed and parallel computing environment.

For a general sparse matrix $A$, however, determining an effective sparsity pattern
of $A^{-1}$ is nontrivial. A-priori sparse patterns may
not capture positions of large entries in $A^{-1}$ effectively, or, they may
capture the positions only when the patterns
are unacceptably dense.  Then the storage becomes a
bottleneck and the time for the construction of the matrix is impractical.
To cope with this difficulty, a number of researchers have proposed adaptive
strategies that start with a simple initial pattern and
successively augment or adaptively adjust this pattern
until $M$ is satisfied with certain accuracy, i.e.,
$\|AM-I\|\leq\varepsilon$ for some norm,
where $\varepsilon$ is fairly small,
or a maximum number of nonzero entries in $M$ is reached.
This idea was first proposed by Cosgrove {\it et al.} \cite{cosgrove92},
and developed by Grote and Huckle \cite{grote1997parallel}, Gould and
Scott \cite{gould98} and Chow and Saad \cite{chow98}. From \cite{benzi1999comparative}
it appears
that the SPAI preconditioning proposed by Grote and
Huckle \cite{grote1997parallel} is more robust
than the one proposed by Chow and Saad \cite{chow98}. One of the key differences
between these procedures is that they use different adaptive ways to
generate sparsity patterns of $M$ by dropping entries of small
magnitude so as to sparsify $M$. Recently, Jia and Zhu \cite{jia2009power}
have proposed a Power Sparse Approximate Inverse (PSAI) procedure that
determines the sparsity pattern of $M$ in a new adaptive way.
Furthermore, they have developed a practical PSAI
algorithm with dropping, called PSAI($tol$), that dynamically drops
the entries in $M$ whose magnitudes are smaller than a prescribed tolerance $tol$
during the process. Extensive numerical experiments in \cite{jiazhang12}
demonstrate that the PSAI($tol$) is at least comparable to
SPAI in \cite{grote1997parallel}.

As is well-known there are three goals for using dropping strategies in the SAI
preconditioning procedure: (i) $M$ should be an effective preconditioner (ii) $M$
should be as sparse as possible so that it is cheap to set up and then to use in a Krylov
solver, when its pattern is determined adaptively, and (iii) $M$ should be
as sparse as possible so as to be cheap to use in a Krylov solver,
when its sparsity pattern is prescribed.
Apparently, dropping is a key step and plays a central role
in designing a robust SAI preconditioning procedure.
Chow \cite{chow2000} suggests a prefiltration strategy and drops the
entries of $A$ itself that are below some tolerance before determining the
pattern of $M$. This prefiltration idea is also adopted in, e.g.,
\cite{Kaporin1990,Kolotilina1988,Tang1999}. Instead of prefiltration,
it may be more effective to apply the sparsification to $M$ after it has
been computed, which is called postfiltration; see, e.g.,
\cite{carpentieri2000,Tang2000}.
Wang and Zhang \cite{wang} have proposed a multistep static SAI preconditioning
procedure that uses both preliftration and postfiltration.
Obviously, for a static SAI procedure,
postfiltration cannot reduce the construction cost of $M$; rather,
it only reduces the application cost of $M$ at each iteration of a Krylov solver.
For an adaptive SAI procedure, a more effective approach is to dynamically drop
entries of small magnitude as they are generated during the construction process.
The approach is more appealing as it makes
$M$ sparse throughout the whole setup
process. As is clear, dropping is more
important for an adaptive SAI procedure than for a static one since
it reduces the setup time of $M$ for the former but not for the latter.
For sparsification applied to FSAI, we refer
the reader to \cite{Bergamaschi2007,Bergamaschi,Ferronato,Kolotilina1999}.

In this paper, we are concerned with dropping tolerance strategies
applied to the adaptive PSAI procedure.
We have noticed that the dropping tolerances used in the literature
are heuristic and empirical. One commonly takes some small quantities,
say $10^{-3}$, as dropping tolerances.
Nevertheless, the mechanism for dropping tolerances is by no means
so simple. Empirically chosen tolerances are not necessarily robust, may not be effective, and
might even lead to failure in preconditioning. Obviously, improperly chosen large
tolerances may lead to a sparser but ineffective $M$, while tolerances that are too
small may lead to a far denser but more effective preconditioner $M$
which is much more time consuming to apply.
Our experiments confirm these statements, and illustrate that simply taking seemingly
small tolerances, as suggested in the literature, may produce a numerical
singular $M$, which can cause a Krylov solver to fail completely. Therefore, dropping
tolerance selection criteria deserve attention and it is desirable to establish a
mathematical theory that can reveal intrinsic relationships between the dropping
tolerances and the quality of $M$. Such selection criteria enable the design of robust
and effective SAI preconditioning procedures.

We point out that dropping has been extensively used in other important
preconditioning techniques such as ILU factorizations \cite{Chow1997,saad1994}.
Some effective selection criteria have been proposed for dropping tolerances in,
e.g., \cite{bollhofer2003robust,Gupta2010,mayer2006alternative}.
It is distinctive that the setup time of good sparse approximate inverses
overwhelms the cost of Krylov solver iterations while this is not necessarily
the case for ILU preconditioners. This is true in a parallel computing
environment, though SPAI and PSAI($tol$) are
inherently parallelizable. Therefore, SAI type preconditioners
are particularly attractive for solving a sequence of linear systems with the
same coefficient matrix, as has been addressed in the literature, e.g.,
\cite{benzi1998bit}, where BiCGStab preconditioned with the adaptive
SPAI algorithm \cite{grote1997parallel} and
the factorized AINV algorithm \cite{benzi1996sparse,benzi1998sparse}
are experimentally shown to be faster than BiCGStab
preconditioned with ILU(0), even in the sequential computing environment
when more than one linear systems is solved.

The goal of this paper is to analyze and establish a rigorous theory for the dropping
tolerance selection criteria used in PSAI. The quality and non-singularity of $M$
obtained by PSAI depends on, and can be very sensitive to, the dropping tolerances.
Based on our theory, we propose an {\em adaptive} dropping criterion that is used to drop
entries of small magnitude dynamically during the setup process of $M$ by PSAI.
The criterion aims to make $M$ as sparse as possible, while possessing
comparable quality to a possibly much denser $M$ obtained by PSAI without dropping.
As a byproduct, the theory applies to static F-norm minimization based
SAI preconditioning procedures, and a similar dropping criterion is derived
that runs postfiltration robustly after $M$ is computed by a static SAI
procedure, making $M$ and its sparsification of comparable preconditioning quality.
As has been noted already, however, as compared to adaptive SAI procedures, dropping in
static SAI procedures does not reduce the setup time of the preconditioner, rather
it reduces the cost of applying the sparser $M$ in the Krylov iteration.

Our numerical experiments illustrate that the dropping tolerance criteria work well
in general, and that the quality and effectiveness of $M$
depends critically on, and is
sensitive to, these criteria. In particular, the reported numerical results demonstrate
that (i) smaller tolerances are not necessary since they may make $M$ denser and
more time consuming to construct, while not offering essential improvements in the
quality of $M$, and (ii) larger tolerances may lead to a numerically singular $M$ so
that preconditioning fails completely.

The paper is organized as follows. In Section 2, we review the Basic PSAI (BPSAI)
procedure without dropping and the PSAI($tol$) procedure with
dropping \cite{jia2009power}. In Section 3, we present results and establish robust
dropping tolerance selection
criteria. In Section 4, we test PSAI($tol$) on a number of real
world problems, justifying our theory and illustrating the robustness
and effectiveness of our selection criterion for dropping tolerances.
We also test the three static F-norm minimization
based SAI procedures with the patterns of $(I+A)^k$,
$(I+|A|+|A^T|)^k$ and $(A^TA)^kA^T$ and
illustrate the effectiveness of our selection criterion for dropping tolerances.
Finally concluding remarks are presented in Section 5.

\section{PSAI algorithms}

The BPSAI procedure is based on F-norm minimization and determines the sparsity
pattern of $M$ adaptively during the process. According to the Cayley--Hamilton
theorem, $A^{-1}$ can be expressed as a matrix polynomial of $A$ of degree $m-1$
with $m\leq n$:
\begin{equation*}
 A^{-1}=\sum_{i=0}^{m-1}c_iA^i
\end{equation*}
with $A^0=I$, the identity matrix, and $c_i,\ i=0,1,\ldots,m-1$, being certain
constants.

Following \cite{jia2009power}, for $i=0,1,\ldots,m-1$, we shall
denote by $A^i(j,k)$ the entry of $A^i$ in position $(j,k)$,
$j,k=1,2,\ldots,n$, and set
$\mathscr J^k_i=\{j|A^i(j,k)\neq 0\}$. For $l=0,1,\ldots,l_{\max}$, define
$\tilde{\mathscr J^k_l}=\cup_{i=0}^l\mathscr J_i^k$.
Let $M=[\tf{m}_1,\tf{m}_2,\ldots,\tf{m}_n]$ be an approximate inverse
of $A$. BPSAI computes each $\tf{m}_k$, $1\leq k\leq n$,
by solving the LS problem
\begin{equation}\label{lsprob}
 \min_{\tf{m}_k(\tilde{\mathscr J_l^k})} \|A(\cdot, \tilde{\mathscr J_l^k})
\tf{m}_k(\tilde{\mathscr J_l^k})-\tf{e}_k\|_2,\ l=0,1,\ldots,l_{\max},
\end{equation}
where $\|\cdot\|_2$ is the vector 2-norm and the matrix spectral norm
and $\tf{e}_k$ is the $k$th column of the $n\times n$ identity matrix $I$.
We exit and output $\tf{m}_k$ when the minimum in (\ref{lsprob})
is less than a prescribed tolerance $\varepsilon$ or $l$ exceeds $l_{\max}$.
We comment that $\tf{m}_k(\tilde{\mathscr J_{l+1}^k})$ can be updated from the
available $\tf{m}_k(\tilde{\mathscr J_l^k})$ very efficiently; see
\cite{jia2009power} for details. The BPSAI procedure is summarized as
Algorithm \ref{alg1}, in which $\tf{a}_k^l$ denotes the $k$th column of $A^l$
and $\tf{a}_k^0=\tf{e}_k$.
It is easily justified that if $l_{\max}$ steps are performed
then the sparsity pattern of $M$ is contained in that of $(I+A)^{l_{\max}}$.

\begin{algorithm}\caption{The BPSAI Algorithm}\label{alg1}
\algsetup{linenodelimiter=.}
For $k=1,2,\ldots,n$, compute $\tf{m}_k$:
 \begin{algorithmic}[1]
\STATE Set $\tf{m}_k=0, l=0$, $\tf{a}_k^0=\tf{e}_k$ and take $\tilde{\mathscr{J}_0^k}=\{k\}$
as the initial sparsity pattern of  $\tf{m}_k$. Choose an accuracy requirement
$\varepsilon$ and the maximum $l_{\max}$ of outer loops.
\STATE Solve (\ref{lsprob}) for $\tf{m}_k$ and let $\tf{r}_k=A\tf{m}_k-\tf{e}_k$.
\WHILE{$\|\tf{r}_k\|_2>\varepsilon$ and $l\leq l_{\max}-1$}
\STATE $\tf{a}_k^{l+1}=A\tf{a}_k^l$, and augment the set $\tilde{\mathscr J_{l+1}^k}$ by
bringing in the indices of the nonzero entries in $\tf{a}_k^{l+1}$.
\STATE $\hat{\mathscr J}=\tilde{\mathscr J_{l+1}^k}\setminus\tilde{\mathscr J_{l}^k}$.
\IF{$\hat{\mathscr J}=\emptyset$}
\STATE Set $l=l+1$, and go to 3;
\ENDIF
\STATE Set $l=l+1$
\STATE Solve (\ref{lsprob}) for updating $\tf{m}_k$ and $\tf{r}_k=A\tf{m}_k-\tf{e}_k$.
\STATE If $\|\tf{r}_k\|\leq\varepsilon$, then break.
\ENDWHILE
  \end{algorithmic}
\end{algorithm}

It is shown in \cite[Theorem 1]{jia2009power} that if $A$ is sparse
{\em irregularly}, that is, there is at least one column of $A$ whose number of
nonzero entries is considerably more than the average number of nonzero entries
per column, then $M$ may become dense very quickly as $l$ increases. However,
when most entries of $A^{-1}$ are small, the corresponding entries
of a good approximate inverse $M$ for $A^{-1}$
are small too, and thus contribute very little to $A^{-1}$.
Therefore, in order to control the sparsity of $M$ and construct an effective
preconditioner, we should apply dropping strategies to BPSAI.
PSAI($tol$) just serves this purpose. It aims to effectively determine an
approximate sparsity pattern of $A^{-1}$ and capture its large entries.
At each while-loop in PSAI($tol$), for the new available $\tf{m}_k$, entries
of small magnitude
below a prescribed tolerance $tol$ are dropped and only large ones are
retained. We describe the PSAI($tol$) algorithm as Algorithm \ref{alg2}, in which the
sparsity pattern of $\tf{m}_k$ is denoted by $\mathscr S_l^k$, $l=0,1,\ldots,l_{\max}$,
which are updated according to steps 9--11 of Algorithm 2.
Hence, for every $k$, we solve the LS problem
\begin{equation}\label{lsprob2}
 \min_{\tf{m}_k(\mathscr S_l^k)} \|A(\cdot, \mathscr S_l^k)
\tf{m}_k(\mathscr S_l^k)-\tf{e}_k\|_2,\ l=0,1,\ldots,l_{\max}.
\end{equation}
Similar to BPSAI, $\tf{m}_k(\mathscr J_{l+1}^k)$ can be updated from the
available $\tf{m}_k(\mathscr J_l^k)$ very efficiently.

\begin{algorithm}[th]\caption{The PSAI($tol$) Algorithm}\label{alg2}
\algsetup{linenodelimiter=.}
For $k=1,2,\ldots,n$, compute $\tf{m}_k$:
 \begin{algorithmic}[1]
\STATE Set $\tf{m}_k=0, l=0$, $\tf{a}_k^0=\tf{e}_k$ and $\mathscr{S}_0^k
=\tilde{\mathscr{J}_0^k}=\{k\}$ as the initial sparsity pattern of $\tf{m}_k$.
Choose an accuracy requirement $\varepsilon$, dropping tolerance $tol$ and
the maximum $l_{\max}$ of outer loops.
\STATE Solve (\ref{lsprob2}) for $\tf{m}_k$ and let $\tf{r}_k=A\tf{m}_k-\tf{e}_k$.
\WHILE{$\|\tf{r}_k\|_2>\varepsilon$ and $l\leq l_{\max}-1$}
\STATE $\tf{a}_k^{l+1}=A\tf{a}_k^l$, and augment the set $\tilde{\mathscr J_{l+1}^k}$
by bringing in the indices of the nonzero entries in $\tf{a}_k^{l+1}$.
\STATE $\hat{\mathscr J}=\tilde{\mathscr J_{l+1}^k}\setminus\tilde{\mathscr J_{l}^k}$.
\IF{$\hat{\mathscr J}=\emptyset$}
\STATE Set $l=l+1$, and go to 3;
\ENDIF
\STATE $\mathscr S_{l+1}^k=\mathscr S_l^k\cup\hat{\mathscr J}$
\STATE Solve (\ref{lsprob2}) for $\tf{m}_k$ and compute $\tf{r}_k=A\tf{m}_k-\tf{e}_k$.
If $\|\tf{r}_k\|\leq\varepsilon$, perform 11 and break.
\STATE Drop the entries of small magnitude in $\tf{m}_k$ whose sizes are below $tol$ and
delete the corresponding indices from $\mathscr S_{l+1}^k$.
\STATE Set $l=l+1$
\ENDWHILE
\end{algorithmic}
\end{algorithm}

From now on we denote by $M$ the preconditioners generated by either BPSAI or
PSAI($tol$). We will
distinguish them by $M$ and $M_d$, respectively when necessary.
The non-singularity and quality of $M$ by BPSAI clearly depends on $\varepsilon$,
while the situation becomes much more complicated for $M_d$.
We will consider these theoretical issues in the next section. At present it
should be clear that for BPSAI the non-singularity and
quality of $M$ is determined by $\varepsilon$ and $l_{\max}$,
two parameters that control while-loop termination in Algorithm \ref{alg1}.
On the one hand, a smaller $\varepsilon$ will generally give rise to higher quality
but possibly denser preconditioner $M$. As a result, more while-loops $l_{\max}$
are used, so that the setup cost of $M$ is higher. We reiterate that it is also more
expensive to apply a denser $M$ at each iteration of a Krylov solver. On the other hand,
a bigger $\varepsilon$ may generate
a sparser but less effective $M$, so that the Krylov solvers use more
iterations to achieve convergence. Unfortunately the selection of $\varepsilon$
can only be empirical. As is standard in the literature, in numerical experiments
we simply take $\varepsilon$
to be a fairly small quantity, say $0.2\sim 0.4$.

\section{Selection criteria for dropping tolerances}

First of all, we should keep in mind  that all SAI
preconditioning procedures are based on the basic hypothesis
that the majority of entries of $A^{-1}$
are small, that is, there exist sparse approximate inverses of $A$.
Mathematically, this amounts to supposing that there exists (at least)
a sparse $M$ such that the residual $\|AM-I\|\leq\varepsilon$ for
some fairly small $\varepsilon$ and some matrix norm $\|\cdot\|$.
The size of $\varepsilon$ is a reasonable measure
for the quality of $M$ as an approximation to
$A^{-1}$. Generally speaking, the smaller $\varepsilon$, the
more accurate $M$ as an approximation of $A^{-1}$.

In the following discussion, we will assume that
BPSAI produces a nonsingular $M$ satisfying $\|AM-I\|\leq\varepsilon$ for some
norm, given fairly small $\varepsilon$. We comment that this is definitely
achieved for a suitable $l_{\max}$. Under this assumption, keep in mind that
$M$ may be relatively dense but have many entries of small magnitude. PSAI($tol$) aims
at dynamically dropping those entries of small magnitude below some absolute dropping
tolerance $tol$ during the setup of $M$ and computing
a new sparser $M$, so as to reduce storage memory and
computational cost of constructing and applying
$M$ as a preconditioner. We are concerned with
two problems. The first problem is how to select $tol$
to make $M$ nonsingular. As will be seen, since $tol$ varies dynamically
for each $k$, $1\leq k\leq n$, as $l$ increases from $0$ to $l_{\max}$ in
Algorithm \ref{alg2}, we will instead denote it by $tol_k$
when computing the $k$th column $\tf{m}_k$. The second is how to
select the $tol_k$ which are required to meet two requirements: (i) $M$ is as
sparse as possible; (ii) its approximation quality is comparable
to that obtained by BPSAI in the sense that the residuals of two $M$
have very comparable sizes. With such sparser $M$,
it is expected that Krylov solvers preconditioned by BPASI and PSAI($tol$),
respectively, will use a comparable number of
iterations to achieve convergence. If so, PSAI($tol$) will
be considerably more effective than BPSAI provided that $M$ obtained
by PSAI($tol$) is considerably sparser than that provided by BPASI.
As far as we are aware, these important problems have not been
studied rigorously and systematically in the context of SAI preconditioning.
The establishment of robust selection criteria, $tol_k$, $k=1,2,\ldots,n$,
for dropping tolerances that meet the two requirements is significant but nontrivial.

Over the years the dropping reported in the literature has been empirical. One
commonly applies a tolerance as follows: set $m_{jk}$ to zero if $|m_{jk}|<tol$, for
some empirical value for $tol$, such as $10^{-3}$, see, e.g.
\cite{chow98,jia2009power,Tang2000,wang}. Due to the absence of mathematical
theory, doing so is problematic, and one may either miss significant entries if $tol$
is too large or retain too many superfluous entries if $tol$ is too small. As a
consequence, $M$ may be of poor quality, or while
a good approximate inverse it may be unduly denser than desirable, leading to
considerably higher setup and application costs.

For general purposes, we should take the size of $\tf{m}_k$
into account when dropping a small entry $m_{jk}$
in $\tf{m}_k$. Define $\tf{f}_k$ to be the $n$-dimensional vector
whose nonzero entries $f_{jk}=m_{jk}$ are those to be dropped in $\tf{m}_k$.
Precisely drop $m_{jk}$ in $\tf{m}_k$ when
\begin{equation}\label{ineq5}
\frac{\|\tf{f}_k\|}{\|\tf{m}_k\|}\leq\mu_k, \quad  k=1,2,\ldots,n
\end{equation}
for some suitable norm $\|\cdot\|$,
where $\mu_k$ is a relative dropping tolerance that is small and
should be chosen carefully based on some mathematical theory.
 For suitably chosen $\mu_k$, our ultimate goal is to derive corresponding
dropping tolerance selection criteria $tol_k$
that are used to adaptively detect and drop small
entries $m_{jk}$ below the tolerance.

In what follows we establish a number of results that
play a vital role in selecting the $\mu_k$ and $tol_k$ effectively.
The matrix norm $\|\cdot\|$ denotes a general induced matrix norm, which includes
the 1-norm and the 2-norm.

\begin{theorem}\label{thm1}
Assume that $\|AM-I\|\leq\varepsilon<1$. Then $M$ is nonsingular. Define
$M_d=M-F$. If $F$ satisfies
\begin{equation}\label{ineq}
 \|F\|<\frac{1-\varepsilon}{\|A\|},
\end{equation}
then $M_d$ is nonsingular.
\end{theorem}
\begin{proof} Suppose that $M$ is singular and let $\tf{w}$  with
$\|\tf{w}\|=1$ be an eigenvector associated with its zero eigenvalue(s), i.e.,
$M\tf{w}=0$. Then for any induced matrix norm we have
$$
\|AM-I\|\geq\|(AM-I)\tf{w}\|=\|\tf{w}\|=1,
$$
a contradiction to the assumption that $\|AM-I\|<1$. So $M$ is nonsingular.

Since
\begin{equation}\label{eq1}
 M_d=M-F=M(I-M^{-1}F),
\end{equation}
from (\ref{ineq}) we have
\begin{equation}\label{ineq2}
\|M^{-1}F\|\leq\|M^{-1}\|\|F\|<(1-\varepsilon)\frac{\|M^{-1}\|}{\|A\|}.
\end{equation}
On the other hand, since
$$
|\|A\|-\|M^{-1}\||\leq\|A-M^{-1}\|\leq\|AM-I\|\|M^{-1}\|\leq\varepsilon\|M^{-1}\|,
$$
we get
$$
(1-\varepsilon)\|M^{-1}\|\leq\|A\|\leq (1+\varepsilon)\|M^{-1}\|,
$$
which means
\begin{equation}\label{ineq1}
1-\varepsilon\leq\frac{\|A\|}{\|M^{-1}\|}\leq
1+\varepsilon.
\end{equation}
Substituting (\ref{ineq1}) into (\ref{ineq2}), we have
\begin{equation*}
 \|M^{-1}F\|<1,
\end{equation*}
from which it follows that $I-M^{-1}F$ in (\ref{eq1}) is nonsingular and so is $M_d$.
\end{proof}

Denote by $M_d$ the sparse approximate inverse of $A$
obtained by PSAI($tol$). Then $M_d$ aims to retain
the entries $m_{jk}$ of large magnitude and drop those of small magnitude
in $M$. The entries of small magnitude to be dropped are those nonzero ones in
the matrix $F$. So, $M_d$ is generally sparser than $M$, and the
number of its nonzero entries is equal to that of $M$ minus that of $F$.

In order to get an $M_d$ comparable to $M$ as an approximation to $A^{-1}$, we
need to impose further restrictions on $F$ and $\varepsilon$, as indicated below.

\begin{theorem}\label{thm2}
Assume that $\|AM-I\|\leq\varepsilon<1$. Then $M$ is nonsingular. Let
$M_d=M-F$. If
\begin{equation}\label{tol_min}
 \|F\|\leq\min\left\{\frac{\varepsilon}{\|A\|},\frac{1-\varepsilon}{\|A\|}\right\},
\end{equation}
then $M_d$ is nonsingular and
\begin{equation}\label{Rd}
 \|AM_d-I\|\leq \min\{1,2\varepsilon\}.
\end{equation}
Specifically, if $\varepsilon<0.5$, then
\begin{equation}\label{tol2}
\|F\|\leq \frac{\varepsilon}{\|A\|}.
\end{equation}
and
\begin{equation}\label{smallrd}
 \|AM_d-I\|\leq 2\varepsilon.
\end{equation}
\end{theorem}

\begin{proof}
The non-singularity of $M$ is already proved in Theorem~\ref{thm1}.
Since $F$ satisfying (\ref{tol_min}) must meet (\ref{ineq}),
the non-singularity of $M_d$ follows from Theorem~\ref{thm1} directly.
From $\|AM-I\|\leq\varepsilon$ and (\ref{tol_min}), we obtain
\begin{eqnarray*}
 \|AM_d-I\|&=&\|AM-AF-I\|\leq\|AM-I\|+\|A\|\|F\|\\
 &\leq&\varepsilon+
 \min\{\varepsilon,1-\varepsilon\}=\min\{1,2\varepsilon\}.
\end{eqnarray*}
(\ref{tol2}) and (\ref{smallrd}) are direct from (\ref{tol_min}) and (\ref{Rd}),
respectively.
\end{proof}

In what follows we always assume that $\varepsilon< 0.5$, so that
(\ref{tol2}) is satisfied and the residual $\|AM_d-I\|\leq 2\varepsilon<1$.
This assumption is purely technical for the brevity and beauty of presentation.
The case that $0.5\leq \varepsilon<1$ can be treated accordingly. The later
theorems can be adapted for this case, but are not considered here.

It is known that $M$ is a good approximation to $A^{-1}$ for a small
$\varepsilon$. This theorem tells us that if dropping tolerances $tol_k$
make $F$ satisfy (\ref{tol2}) then the $M_d$
and $M$ have comparable residuals and are approximate inverses of $A$ with
comparable accuracy, provided that $\varepsilon$ is fairly small.
In this case, we claim that they possess a similar preconditioning quality
for a Krylov solver, and it is expected that the Krylov solver preconditioned by
$M_d$ and $M$, respectively, use a comparable number of iterations to achieve convergence.

In the above results, the assumptions and bounds are determined by matrix norms,
which are thus not directly applicable to meet our goals. To be more practical,
we now present a theorem under the assumption that
$\|A\tf{m}_k-\tf{e}_k\|\leq\varepsilon$ for $k=1,2,\ldots,n$, which
is just the stopping criterion in Algorithms~\ref{alg1}--\ref{alg2}
and the SPAI algorithm \cite{grote1997parallel}, etc., where the norm is the 2-norm.

\begin{theorem}\label{thm3}
For a given vector norm $\|\cdot\|$, let $M=[\tf{m}_1,\tf{m}_2,\ldots,\tf{m}_n]$
satisfy $\|A\tf{m}_k-\tf{e}_k\|\leq\varepsilon< 0.5$ for
$k=1,2,\ldots,n$ and  let $M_d=M-F=[\tf{m}^d_1,\tf{m}^d_2,\ldots,\tf{m}^d_n]$
with $F=[\tf{f}_1,\tf{f}_2,\ldots,\tf{f}_n]$.
If
\begin{equation}\label{fksmall}
\|\tf{f}_k\|\leq\frac{\varepsilon}{\|A\|},\ k=1,2,\ldots,n,
\end{equation}
then
\begin{equation}\label{2eps}
 \|A\tf{m}_k^d-\tf{e}_k\|\leq2\varepsilon.
\end{equation}
\end{theorem}

\begin{proof}
 Let $\tf{r}_k=A\tf{m}_k-\tf{e}_k$.
 Then from $\tf{m}_k^d=\tf{m}_k-\tf{f}_k$ we get
\begin{equation*}
\begin{split}
 \|A\tf{m}_k^d-\tf{e}_k\|&=\|\tf{r}_k-A\tf{f}_k\|\leq\|\tf{r}_k\|+\|A\tf{f}_k\|\\
&\leq\varepsilon+\|A\tf{f}_k\|\leq\varepsilon+\|A\|\|\tf{f}_k\|,
\end{split}
\end{equation*}
from which, with the assumption of the theorem, (\ref{2eps}) holds.
\end{proof}

Still, this theorem does not fit nicely for our use. For the later theoretical and
practical background, we present a mixed norm result, which
is a variant of Theorem~\ref{thm3}.

\begin{theorem}\label{thm5}
Let $M=[\tf{m}_1,\tf{m}_2,\ldots,\tf{m}_n]$ satisfy $\|A\tf{m}_k-\tf{e}_k\|_2
\leq\varepsilon< 0.5$ for
$k=1,2,\ldots,n$ and  let $M_d=M-F=[\tf{m}^d_1,\tf{m}^d_2,\ldots,\tf{m}^d_n]$
with $F=[\tf{f}_1,\tf{f}_2,\ldots,\tf{f}_n]$.
If
\begin{equation}\label{fksmallnorm1}
\|\tf{f}_k\|_1\leq\frac{\varepsilon}{\|A\|_1},\ k=1,2,\ldots,n,
\end{equation}
then
\begin{equation}\label{2epsnorm2}
 \|A\tf{m}_k^d-\tf{e}_k\|_2\leq2\varepsilon.
\end{equation}
\end{theorem}

\begin{proof}
 Let $\tf{r}_k=A\tf{m}_k-\tf{e}_k$.
 Then from $\tf{m}_k^d=\tf{m}_k-\tf{f}_k$ we get
\begin{equation*}
\begin{split}
 \|A\tf{m}_k^d-\tf{e}_k\|_2&=\|\tf{r}_k-A\tf{f}_k\|_2\leq\|\tf{r}_k\|_2+\|A\tf{f}_k\|_2
 \leq\|\tf{r}_k\|_2+\|A\tf{f}_k\|_1\\
&\leq\varepsilon+\|A\tf{f}_k\|_1\leq\varepsilon+\|A\|_1\|\tf{f}_k\|_1,
\end{split}
\end{equation*}
from which, with the assumption (\ref{fksmallnorm1}), (\ref{2epsnorm2}) holds.
\end{proof}

Theorem~\ref{thm5} cannot guarantee that $M$ and $M_d$
are nonsingular. In \cite{grote1997parallel}, Grote and Huckle have presented
some theoretical properties of a sparse approximate inverse. Particularly,
for the matrix 1-norm, Theorem 3.1 and Corollary 3.1 of \cite{grote1997parallel}
read as follows when applied to $M$ and $M_d$ defined by Theorem~\ref{thm5}.

\begin{theorem}\label{grotehuckle}
Let $\tf{r}_k=A\tf{m}_k-\tf{e}_k$, $\tf{r}_k^d=A\tf{m}_k^d-\tf{e}_k$ and
$p=\max_{1\leq k\leq n}\{\mbox{ the number of nonzero entries of $\tf{r}_k$}\}$,
$p_d=\max_{1\leq k\leq n}\{\mbox{ the number of nonzero entries of $\tf{r}_k^d$}\}$.
Then if $\|\tf{r}_k\|_2\leq\varepsilon$ and $\|\tf{r}_k^d\|_2\leq2\varepsilon,
\,k=1,2,\ldots,n$, we have
\begin{eqnarray}
\|AM-I\|_1&\leq& \sqrt{p}\varepsilon, \label{mres}\\
\|AM_d-I\|_1&\leq& 2\sqrt{p_d}\varepsilon. \label{mdres}
\end{eqnarray}
Furthermore, if $\sqrt{p}\varepsilon<1$ and
$2\sqrt{p_d}\varepsilon<1$, respectively, $M$ and $M_d$ are nonsingular and
\begin{eqnarray}
\frac{\|M-A^{-1}\|_1}{\|A^{-1}\|_1}&\leq& \sqrt{p}\varepsilon,\label{merror}\\
\frac{\|M_d-A^{-1}\|_1}{\|A^{-1}\|_1}&\leq&2\sqrt{p_d}\varepsilon.\label{mderror}
\end{eqnarray}
\end{theorem}

Theorem~\ref{thm5} indicates that given $\varepsilon< 0.5$ and $l_{\max}$,
if the while-loop in BPSAI terminates due to $\|\tf{r}_k\|\leq\varepsilon$
for all $k$ and dropping tolerance $tol_k$ is selected
such that (\ref{fksmallnorm1})
holds, then the corresponding columns of $M_d$ and $M$ are of similar quality
provided that $\varepsilon$ is fairly small. It is noted
in \cite{grote1997parallel} for the SPAI that $p$ is usually
much smaller than $n$. This is also the case for BPSAI and PSAI($tol$).
However, we should realize that such a sufficient condition is very
conservative, as pointed out in \cite{grote1997parallel}.
In practice, for a rather mildly small $\varepsilon$, say $0.3$,
$M$ is rarely singular.

Theorem~\ref{grotehuckle} shows that $M_d$ and $M$ are approximate
inverses of $A$ with similar
accuracy and are expected to have a similar preconditioning quality.
Besides, since $\tf{m}_k$ is generally denser than $\tf{m}_k^d$,
$\tf{r}_k$ is heuristically denser than $\tf{r}_k^d$, i.e.,
$p_d$ is more than likely to be smaller than $p$. Consequently, $2\sqrt{p_d}$ is
comparable to $\sqrt{p}$. This means that the bounds for $M_d$ are close to and
furthermore may not be bigger than the corresponding ones for $M$
in Theorem~\ref{grotehuckle}, so $M_d$ and $M$ are approximations to
$A^{-1}$ with very similar accuracy or quality.

Theorems~\ref{thm2}--\ref{grotehuckle} are fundamental and
relate the quality of $M_d$ to that of $M$ in terms
of $\varepsilon$ quantitatively and explicitly.
They provide necessary ingredients for reasonably selecting
relative dropping tolerance $\mu_k$ in (\ref{ineq5})
to get a possibly much sparser preconditioner
$M_d$ that has a similar preconditioning quality to $M$.
In what follows we present a detailed analysis and propose robust selection criteria
for dropping tolerance $tol_k$.

For given $l_{\max}$, suppose that $M$ obtained by BPSAI is nonsingular
and satisfies $\|A\tf{m}_k-\tf{e}_k\|\leq\varepsilon< 0.5$
for $k=1,2,\ldots,n$. To achieve our goal, the crucial point is
how to combine (\ref{ineq5}) with condition (\ref{fksmallnorm1}) in
Theorem~\ref{thm5} in an organic
and reasonable way. For the 1-norm in (\ref{ineq5}), a unification of
(\ref{ineq5}) and (\ref{fksmallnorm1}) means that
\begin{equation}\label{ineq6}
 \|\tf{f}_k\|_1\leq\mu_k\|\tf{m}_k\|_1 \mbox{ and } \|\tf{f}_k\|_1\leq
 \frac{\varepsilon}{\|A\|_1}
\end{equation}
at every while-loop in PSAI($tol$), where the bound in the first relation is
to be determined and the bound in the second relation
is given explicitly. This is the starting point for the analysis determining
how to drop the small entries $m_{jk}$ in $\tf{m}_k$.


Before proceeding, supposing that $\mu_k$ is given in a disguise, we investigate how to
choose $\tf{f}_k$ to make (\ref{ineq6}) hold. Obviously, it suffices to
drop nonzero $m_{jk}$, $1\leq j\leq n$ in $\tf{m}_k$ as long as its size is no more than
the bounds in (\ref{ineq6}) divided by $nnz(\tf{f}_k)$.
Since $nnz(\tf{f}_k)$ is not known a-priori, in practice we replace it by the currently
available $nnz(\tf{m}_k)$ before dropping, which is an {\em upper bound}
for $nnz(\tf{f}_k)$. Therefore, we should drop an $m_{jk}$ when
it satisfies
\begin{equation}\label{tolk}
|m_{jk}|\leq\frac{\mu_k\|\tf{m}_k\|_1}{nnz(\tf{m}_k)}
\mbox{ and } |m_{jk}|\leq \frac{\varepsilon}
{nnz(\tf{m}_k)\|A\|_1},
\ \ k=1,2,\ldots,n.
\end{equation}
Given (\ref{ineq6}), we comment that each of the above bounds may be
correspondingly conservative as
$nnz(\tf{m}_k)>nnz(\tf{f}_k)$. But it seems hard, if not impossible,
to replace the unknown $nnz(\tf{f}_k)$ by any other better computable
estimate than $nnz(\tf{m}_k)$.

Next we go to our central concern and discuss how to relate $\mu_k$ to
$\varepsilon$ so as to establish a robust selection criterion $tol_k$
for dropping tolerances. Precisely, as (\ref{tolk}) has indicated,
we aim at selecting suitable relative tolerance $\mu_k$ and then
drop entries of small magnitude in $\tf{m}_k$
below $\frac{\mu_k\|\tf{m}_k\|_1}{nnz(\tf{m}_k)}$.
By Theorems~\ref{thm5}--\ref{grotehuckle}, the second bound in (\ref{ineq6})
and its induced bound in (\ref{tolk})
serve to guarantee that $M_d$ has comparable
preconditioning quality to $M$. Therefore, if an $m_{jk}$ satisfies the second
relation in (\ref{tolk}), it should be dropped. Otherwise, if $\mu_k$ satisfies
$$
\mu_k\|\tf{m}_k\|_1>\frac{\varepsilon}{\|A\|_1}
$$
and we use the dropping criterion
$$
tol_k=\frac{\mu_k\|\tf{m}_k\|_1}{nnz(\tf{m}_k)}>
\frac{\varepsilon}{nnz(\tf{m}_k)\|A\|_1}
$$
for some $k$,
we would possibly drop an excessive number of nonzero entries and $M_d$ would be too
sparse. The resulting $M_d$ may mean that (\ref{fksmallnorm1}) is not satisfied and that
$M_d$ is a poor quality preconditioner, possibly also numerically singular, which could
lead to a complete failure of the Krylov solver. Thus larger $tol_k$ should not be selected.

On the other hand, if we chose $\mu_k$ such that
$$
\mu_k\|\tf{m}_k\|_1<\frac{\varepsilon}{\|A\|_1}
$$
and took the dropping criterion
$$
tol_k=\frac{\mu_k\|\tf{m}_k\|_1}{nnz(\tf{m}_k)}<
\frac{\varepsilon}{nnz(\tf{m}_k)\|A\|_1},
$$
Theorem~\ref{thm5} would hold
and the preconditioning quality of $M_d$ would be guaranteed and comparable
to that of $M$. However, Theorems~\ref{thm5}--\ref{grotehuckle}
show that the accuracy of such $M_d$ cannot be improved as approximate
inverses of $A$ as $\mu_k$ and $tol_k$ become smaller. Computationally,
it is crucial to realize that the smaller $tol_k$, generally
the denser $M_d$, leading to an increased setup cost for $M_d$ and
more expensive application of $M_d$ in a Krylov iteration. As a consequence,
such smaller $tol_k$ are not desirable and
may lower the efficiency of constructing $M_d$. Consequently such smaller values for
$tol_k$ should be abandoned.

In view of the above analysis, it is imperative that we find
an optimal balance point. Our above arguments
have suggested an optimal and most effective choice for $\mu_k$:
we should select $\mu_k$ to make two bounds in (\ref{ineq6}) equal:
\begin{equation}\label{mutol}
\mu_k\|\tf{m}_k\|_1=\frac{\varepsilon}{\|A\|_1}.
\end{equation}
From (\ref{tolk}), this selection leads to our ultimate dropping criterion
\begin{equation}\label{muepsilon2}
tol_k=\frac{\varepsilon}{nnz(\tf{m}_k)\|A\|_1}.
\end{equation}

We point out that since (\ref{fksmallnorm1}) is a sufficient but
not necessary condition for (\ref{2epsnorm2}), $tol_k$ defined above is also
sufficient but not necessary for (\ref{2epsnorm2}).
Also, it may be conservative since we replace the smaller true value $nnz(\tf{f}_k)$
by its upper bound $nnz(\tf{m}_k)$ in the denominator. As a result,
$tol_k$ may be considerably smaller than it
should be in an ideal case. We should note that
$\mu_k$ and $tol_k$ are varying parameters during the while-loop
in Algorithm 2 as $nnz(\tf{m}_k)$ changes when the while-loop $l$ increases
from $0$ to $l_{\max}$.

In the literature one commonly uses {\em fixed} dropping
tolerance $tol$ when constructing a SAI preconditioner $M$,
which is, empirically and heuristically, taken as some seemingly
small quantity, say $10^{-2}$, $10^{-3}$ or $10^{-4}$,
without taking $\varepsilon$ into consideration; see, e.g.,
\cite{chow98,jia2009power,Tang2000,wang}. Our theory has indicated that
the non-singularity and preconditioning quality of $M_d$ is critically
dependent and possibly sensitive to the choice of the dropping tolerance.
For fixed tolerances that are larger than that defined by
(\ref{muepsilon2}) for some $k$ during the construction of $M_d$,
we report numerical experiments that indicate that the resulting $M_d$
obtained by PSAI($tol$) can be {\em exactly singular in finite precision arithmetic}.
We also report experiments that show decreasing such large tolerances
by one order of magnitude, can provide high quality and nonsingular $M_d$.
Thus, the robustness and effectiveness of $M_d$ depends directly on the tolerance.

We stress that Theorems~\ref{thm1}--\ref{grotehuckle} hold for a generally given
approximate inverse $M$ of $A$ and do not depend on a specific F-norm minimization
based SAI preconditioning procedure. Note that, for all the static F-norm
minimization based SAI preconditioning procedures,
the high quality $M$ constructed from $A$ itself are often quite dense and
their applications in Krylov solvers can be time consuming. To
improve the overall performance of solving $Ax=b$, one often sparsifies
$M$ after its computation, by using postfiltration on $M$
to obtain a new sparser approximate inverse $M_d$ \cite{carpentieri2000,Tang2000}.
However, as already stated in the introduction, postfiltration itself cannot
reduce the cost of constructing $M$ but can reduce the cost of applying $M$ in Krylov
iterations.

As a byproduct, our theory can be very easily adapted to
a static F-norm minimization based SAI preconditioning procedure.
The difference and simplification is that,
for a static SAI procedure, $\mu_k$ in (\ref{muepsilon2}) and $tol_k$ are fixed for
each $k$ as $\tf{m}_k$ and $nnz(\tf{m}_k)$ are already determined a-priori before
dropping is performed on $M$. Practically, after computing $M$ by a static SAI
procedure, we record $\|A\tf{m}_k-\tf{e}_k\|=\varepsilon_k,\
k=1,2,\ldots,n$ and compute the {\em constants} $nnz(\tf{m}_k)$ for $k=1,2,\ldots,n$.
Assume that $\varepsilon_k<0.5,\ k=1,2,\ldots, n$.
Then by (\ref{tolk}) and (\ref{muepsilon2}) we drop $m_{jk}$ whenever
\begin{equation}\label{staticcriterion}
\mid m_{jk}\mid\leq tol_k=\frac{\varepsilon_k}{nnz(\tf{m}_k)\|A\|_1},
\ j=1,2,\ldots,n.
\end{equation}
In such a way, based on Theorem~\ref{thm5}
we get a new sparser approximate inverse $M_d$ whose $k$th
column $\tf{m}_k^d$ satisfies $\|A\tf{m}_k^d-\tf{e}_k\|\leq 2\varepsilon_k$.
Define $\varepsilon=\max_{k=1,2,\ldots,n}\varepsilon_k$. Then
Theorem~\ref{grotehuckle} holds. So $M_d$ has a similar preconditioning
quality to the generally denser $M$ obtained by the static SAI procedure
without dropping. We reiterate, however, that in contrast to adaptive PSAI($tol$)
 where small entries below a tolerance are dropped immediately when
 they are generated during the while loop of Algorithm 2,
the static SAI procedure does not reduce the setup cost of $M_d$  since it
performs sparsification only after computation of $M$. There is
relatively greater benefit in dropping in adaptive SAI preconditioning.

\section{Numerical experiments}\label{numexp}

In this section we test a number of real
world problems coming from scientific and engineering applications,
which are described in Table \ref{table1}\footnote{All of these matrices are
from the Matrix Market of the National Institute of Standards and Technology at
http://math.nist.gov/MatrixMarket or from the University of Florida Sparse Matrix
Collection at http://www.cise.ufl.edu/research/sparse/matrices/.}. We shall
demonstrate the robustness and effectiveness of our selection criteria
for dropping tolerances applied to PSAI($tol$) and, as a byproduct,
three F-norm minimization based static SAI preconditioning procedures.

The numerical experiments are performed on an Intel(R) Core (TM)2Duo Quad CPU E8400
@ $3.00$GHz processor with main memory 2 GB using Matlab 7.8.0 with the machine
precision $\epsilon_{\rm mach}=2.22\times10^{-16}$ under the Linux operating system.
Preconditioning is from the right except pores\_2, for which we
found that left preconditioning outperforms right preconditioning very considerably.
It appears that the rows of pores\_2s inverse can be
approximated more effectively
than its columns by PSAI($tol$). Krylov solvers employed are BiCGStab and
the restarted GMRES(50) algorithms \cite{barrett1994templates}, and we use
the codes from Matlab 7.8.0. We comment that if the output of iterations
for the code {\sf BiCGStab.m} is $k$, the dimension of the Krylov subspace is $2k$
and BiCGStab performs $2k$ matrix-vector products.
The initial guess is always $x_0=0$, and the
right-hand side $b$ is formed by choosing the solution $x=[1,1,\ldots,1]^T$.
The stopping criterion is
\begin{equation*}
 \frac{\|b-Ax_m\|_2}{\|b\|_2}<10^{-8}, \quad x_m=My_m,
\end{equation*}
where $y_m$ is the approximate solution obtained by BiCGStab or
GMRES(50) applied to the preconditioned linear system $AMy=b$.
We run all the algorithms in a sequential environment. We will observe
that the setup cost for $M$ dominates the entire cost of solving $Ax=b$.
As stressed in the introduction, this is a distinctive feature of
SAI preconditioning procedures even in a distributed parallel environment.

\begin{table}[ht]
\centering{
\caption{The description of test matrices ($n$ is the order of a matrix; $nnz$ is
the number of nonzero entries)}\label{table1}
\begin{tabular}{lccl}
 \hline\noalign{\smallskip}
Matrix&$n$&$nnz$&Description\\
\noalign{\smallskip}\hline\noalign{\smallskip}
epb1      &  14734 & 95053 & Plate-fin heat exchanger\\
fidap024 & 2283 & 48733 & Computational fluid dynamics problem\\
fidap028 & 2603 & 77653 & Computational fluid dynamics problem\\
fidap031 & 3909 & 115299& Computational fluid dynamics problem\\
fidap036 & 3079 & 53851 & Computational fluid dynamics problem\\
nos3 &          960 &  8402 & Biharmonic equation\\
nos6   &        675&   1965 & Poisson equation\\
orsreg\_1   &   2205 &  14133& Oil reservoir simulation. Jacobian Matrix\\
orsirr\_1    &  1030 &   6858  & As ORSREG1, but unnecessary cells coalesced\\
orsirr\_2    &    886 &  5970 & As ORSIRR1, with further coarsening of grid\\
pores\_2    &   1224   & 9613 & Reservoir simulation\\
sherman1    & 1000 &   3750 & Oil reservoir simulation $10\times10\times10$grid\\
sherman2  &   1080 &  23094& Oil reservoir simulation $6\times6\times5$ grid\\
sherman3   &  5005 &  20033& Oil reservoir simulation $35\times11\times13$ grid\\
sherman4   &  1104 &   3786 & Oil reservoir simulation $16\times23\times3$ grid\\
sherman5   &  3312  & 20793& Oil reservoir simulation $16\times23\times3$ grid\\
\noalign{\smallskip}\hline
\end{tabular}}.
\end{table}
\par
In the experiments, we take different $\varepsilon$ and suitably small integer
$l_{\max}$ so as to control the quality of $M$ in Algorithms~\ref{alg1}--\ref{alg2},
i.e., the BPSAI and PSAI($tol$) algorithms, in which the while-loop terminates
when $\|A\tf{m}_k-\tf{e}_k\|_2\leq\varepsilon$
or $l> l_{\max}$.
In all the tables, we use the following notations:
\begin{itemize}
\item $\varepsilon$: the accuracy requirements in Algorithms~\ref{alg1}--\ref{alg2};

\item $l_{\max}$: the maximum while-loops that Algorithms~\ref{alg1}--\ref{alg2}
allow;


\item $iter\_\,b$ and $iter\_\,g$: the iteration numbers of BiCGStab
and GMRES(50), respectively;

\item $spar=\frac{nnz(M)}{nnz(A)}$: the sparsity of
$M$ relative to $A$;

\item $mintol$ and $maxtol$: the minimum and maximum of $tol_k$ defined
by (\ref{muepsilon2}) for $k=1,2,\ldots,n$ and $l=0,1,\ldots,l_{\max}$;

\item $ptime$: the setup time (in second) of $M$;

\item $r_{\max}=\max_{k=1,\ldots,n}\|A\tf{m}_k-\tf{e}_k\|$;

\item $coln$: the number of columns of $M$ that fail to meet the accuracy
requirement $\varepsilon$;

\item $\dagger$: flags convergence not attained within 1000 iterations.

\end{itemize}

We report the results in Tables~\ref{table2}--\ref{static3}.
Our aims are four fold: (i) our selection
criterion (\ref{muepsilon2}) for $tol_k$ works very robustly and effectively
since Krylov solvers preconditioned by PSAI($tol$) and BPSAI use
almost the same iterations, the $tol_k$ smaller than those defined by (\ref{muepsilon2})
are not necessary, rather they increase the total cost of solving linear systems
since they do not improve the preconditioning quality
of $M_d$, increase the setup time of $M_d$ and make $M_d$
become denser. (ii) the quality of $M_d$ depends
on the choice of $tol_k$ critically and an empirically chosen fixed small $tol_k$ may produce
a numerically singular $M_d$. (iii) $tol_k$ of one order smaller than those
in case (ii) may dramatically improve the preconditioning effectiveness of $M_d$.
This means that an empirically chosen $tol$ may fail to produce a good preconditioner.
(iv) As a byproduct, we show that the selection criterion (\ref{staticcriterion})
for $tol_k$ works well for static F-norm minimization SAI
preconditioning procedures with three common prescribed patterns.
We present the results on (i)--(iii) in subsection 4.1 and the results on (iv)
in subsection 4.2, respectively.

\subsection{Results for PSAI($tol$)}

We shall illustrate that our dropping criterion (\ref{muepsilon2}) for $tol_k$
is robust for various parameters $\varepsilon$ and $l_{\max}$.
We will show that for a smaller $\varepsilon$ we need more while loops, and
resulting $M_d$ are denser and cost more to construct, but are more effective for
 accelerating BiCGStab and GMRES(50), that is,
the Krylov solvers use fewer iterations to achieve convergence. We also show
that for fairly small $\varepsilon=0.2,0.3,0.4$, Algorithms~\ref{alg1}--\ref{alg2}
can compute a good sparse approximation $M$ of $A^{-1}$
 with accuracy $\varepsilon$ for small integer $l_{\max}$,
and the maximum $l_{\max}=11$ is needed for $\varepsilon=0.2$.

We summarize the results obtained by the two Krylov solvers with and without PSAI($tol$)
preconditioning in Table~\ref{table2}. We see that the two Krylov solvers without
preconditioning failed to solve most test problems
within 1000 iterations while two Krylov solvers
are accelerated by PSAI($tol$) preconditioning substantially and they
solved all the problems quite successfully except for $\varepsilon=0.4$ and $l_{\max}=5,8$,
where GMRES(50) did not converge for fidap024, fidap036 and sherman3. Particularly,
the Krylov solvers preconditioned by PSAI($tol$) solved sherman2 very quickly
and converged within 10 iterations for three given $\varepsilon=0.2, 0.3, 0.4$, but
they failed to solve the problem when no preconditioning is used.

\begin{table}[!htb]
\centering
 \caption{Convergence results for all the test problems: unpreconditioned ($M=I$)
and PSAI($tol$) procedure with different $\varepsilon$ and $l_{max}$.
Note: when the iterations for BiCGStab are $k$, the dimension of the Krylov
subspace is $2k$.
}\label{table2}
\begin{tabular}{@{}cr@{,~}lcc@{~~}c@{~~}r@{,~}l@{~~}c@{~~}ccc@{~~}c@{~~}r@{,~}l@{~~}c@{~~}c@{}}
\hline\noalign{\smallskip}
&\multicolumn{2}{c}{$M=I$}&&\multicolumn{6}{c}{PSAI($tol$),~$\varepsilon=0.2,l_{\max}=8$}
&&\multicolumn{6}{c}{PSAI($tol$),~$\varepsilon=0.2,l_{\max}=11$}\\
\cline{2-3}\cline{5-10}\cline{12-17}
Matrix & $iter\_b$&$iter\_\,g$&& $spar$&$ptime$ & $iter\_\,b$&$iter\_\,g$&$r_{\max}$&$coln$
&& $spar$&$ptime$ & $iter\_\,b$&$iter\_\,g$&$r_{\max}$&$coln$\\
\noalign{\smallskip}\hline\noalign{\smallskip}
 epb1 & 433&$\dagger$ &&3.20 &112.36 & 120&272 & 0.20 &0&&3.20 &111.36 & 120&272 & 0.20 &0 \\
fidap024 & $\dagger$&$\dagger$ &&8.80 &121.52& 27&40 & 0.20 &0&&8.80 &121.52 & 27&40 & 0.20 &0 \\
fidap028 & $\dagger$&$\dagger$ &&9.97 &423.75 & 31&42 & 0.26 &29&&10.11 &437.36 & 31&41 & 0.20 &0 \\
fidap031 & $\dagger$&$\dagger$ &&6.40 &267.74 & 58&103 & 0.35 &1&&6.41 &269.70 & 58&102 & 0.20 &0 \\
fidap036 & $\dagger$&$\dagger$ &&5.78 &63.88 & 34&48 & 0.20 &0&&5.78 &63.88 & 34&48 & 0.20 &0 \\
nos3 & 213& $\dagger$ &&3.89 &3.72 & 49&98 & 0.20&0&&3.89 &3.72 & 49&98 & 0.20&0 \\
nos6 & $\dagger$&$\dagger$ &&2.73 &0.48 & 19&24 & 0.20&0&&2.73 &0.48 & 19&24 & 0.20&0 \\
orsirr\_1 & $\dagger$&$\dagger$ &&10.15 &7.41 & 15&26 & 0.20 &0&&10.15 &7.41 & 15&26 & 0.20 &0\\
orsirr\_2 & $\dagger$&$\dagger$ &&10.71 &6.53 & 16&25 & 0.20 &0&&10.71 &6.53 & 16&25 & 0.20 &0 \\
orsreg\_1 & 687&346 &&9.16 &19.49 & 18&29 & 0.20 &0&&9.16 &19.49 & 18&29 & 0.20 &0\\
pores\_2 & $\dagger$&$\dagger$&& 17.41&26.38&19&26&0.27&15&&17.66&27.53&19&27&0.20&0 \\
sherman1 & 356&$\dagger$&&6.54 &1.37 & 18&28 & 0.27 &2&&6.58 &1.38 & 18&28 & 0.20 &0 \\
sherman2 &$\dagger$&$\dagger$&&3.40 &6.58 & 4&6 & 0.20 &0&&3.40 &6.58 & 4&6 & 0.20 &0\\
sherman3 & $\dagger$&$\dagger$&&4.86 &10.52 & 81&229 & 0.32 &32&&4.90 &10.71 & 81&228 & 0.20 &0 \\
sherman4 & 101&377 &&3.36 &0.76 & 24&34 & 0.20 &0&&3.36 &0.76 & 24&34 & 0.20 &0 \\
sherman5 & $\dagger$&$\dagger$ &&3.34 &4.89 & 21&30 & 0.20 &0&&3.34 &4.89 & 21&30 & 0.20 &0 \\ \noalign{\smallskip}\hline
\end{tabular}
\vskip .3cm
\begin{tabular}{@{}cr@{,~}lcc@{~~}c@{~~}r@{,~}l@{~~}c@{~~}ccc@{~~}c@{~~}r@{,~}l@{~~}c@{~~}c@{}}
\hline\noalign{\smallskip}
&\multicolumn{2}{c}{$M=I$}&&\multicolumn{6}{c}{PSAI($tol$),~$\varepsilon=0.3,l_{\max}=6$}
&&\multicolumn{6}{c}{PSAI($tol$),~$\varepsilon=0.3,l_{\max}=10$}\\
\cline{2-3}\cline{5-10}\cline{12-17}
Matrix & $iter\_b$&$iter\_\,g$&& $spar$&$ptime$ & $iter\_\,b$&$iter\_\,g$&$r_{\max}$&$coln$
&& $spar$&$ptime$ & $iter\_\,b$&$iter\_\,g$&$r_{\max}$&$coln$\\
\noalign{\smallskip}\hline\noalign{\smallskip}
 epb1 & 433&$\dagger$ &&1.17 &36.71 & 170&408 & 0.30 &0&&1.17 &36.71 & 170&408 & 0.30 &0 \\
fidap024 & $\dagger$&$\dagger$ &&5.22 &34.52 & 46&98 & 0.38 &12&&5.27 &34.91 & 46&97 & 0.30 &0 \\
fidap028 & $\dagger$&$\dagger$ &&5.48 &113.09 & 64&168 & 0.33 &10&&5.50 &117.28 & 64&159 & 0.30 &0 \\
fidap031 & $\dagger$&$\dagger$ &&3.08 &58.39 & 104&387 & 0.56 &2&&3.09 &59.18 & 104&444 & 0.30 &0 \\
fidap036 & $\dagger$&$\dagger$ &&2.51 &12.73 & 69&119 & 0.30 &0&&2.51 &12.73 & 69&119 & 0.30 &0 \\
nos3 & 213& $\dagger$ &&1.65 &1.29 & 69&144 & 0.30 &0&&1.65 &1.29 & 69&144 & 0.30 &0 \\
nos6 & $\dagger$&$\dagger$ &&0.94 &0.20 & 35&37 & 0.30 &0&&0.94 &0.20 & 35&37 & 0.30 &0 \\
orsirr\_1 & $\dagger$&$\dagger$ &&5.36 &3.46 & 25&37 & 0.30 &0&&5.36 &3.46 & 25&37 & 0.30 &0 \\
orsirr\_2 & $\dagger$&$\dagger$ &&5.66 &3.05 & 23&36 & 0.30 &0&&5.66 &3.05 & 23&36 & 0.30 &0 \\
orsreg\_1 & 687&346 &&4.02 &7.31 & 27&47 & 0.30 &0&&4.02 &7.31 & 27&47 & 0.30 &0 \\
pores\_2 & $\dagger$&$\dagger$&& 8.67 & 6.84  & 37& 51&0.51&12&&8.78&7.31&37&50&0.30&0\\
sherman1 & 356&$\dagger$&&2.86 &0.70 & 27&40 & 0.38 &2&&2.89 &0.74 & 27&40 & 0.30 &0 \\
sherman2 &$\dagger$&$\dagger$&&2.74 &4.54 & 4&7 & 0.30 &0&&2.74 &4.54 & 4&7 & 0.30 &0 \\
sherman3 & $\dagger$&$\dagger$&&1.93 &4.89 & 145&627 & 0.35 &34&&1.96 &5.07 & 143&900 & 0.30 &0 \\
sherman4 & 101&377 &&1.25 &0.35 & 34&49 & 0.30 &0&&1.25 &0.35 & 34&49 & 0.30 &0 \\
sherman5 & $\dagger$&$\dagger$ &&1.57 &2.05 & 29&43 & 0.30 &0&&1.57 &2.05 & 29&43 & 0.30 &0 \\
\noalign{\smallskip}\hline
\end{tabular}
\vskip .3cm
\begin{tabular}{@{}cr@{,~}lcc@{~~}c@{~~}r@{,~}l@{~~}c@{~~}ccc@{~~}c@{~~}r@{,~}l@{~~}c@{~~}c@{}}
\hline\noalign{\smallskip}
&\multicolumn{2}{c}{$M=I$}&&\multicolumn{6}{c}{PSAI($tol$),~$\varepsilon=0.4,l_{\max}=5$}
&&\multicolumn{6}{c}{PSAI($tol$),~$\varepsilon=0.4,l_{\max}=8$}\\
\cline{2-3}\cline{5-10}\cline{12-17}
Matrix & $iter\_b$&$iter\_\,g$&& $spar$&$ptime$ & $iter\_\,b$&$iter\_\,g$&$r_{\max}$&$coln$
&& $spar$&$ptime$ & $iter\_\,b$&$iter\_\,g$&$r_{\max}$&$coln$\\
\noalign{\smallskip}\hline\noalign{\smallskip}
 epb1 & 433&$\dagger$ &&0.60 &22.23 & 237&474 & 0.40 &0&&0.60 &22.23 & 237&474
 & 0.40 &0 \\
fidap024 & $\dagger$&$\dagger$ &&3.26 &12.77 & 95&$\dagger$  & 0.42 &6&&3.28 &11.54 & 91
&$\dagger$ & 0.40 &0 \\
fidap028 & $\dagger$&$\dagger$ &&3.33 &37.70 & 99&299 & 0.40 &0&&3.33 &37.70 & 99&299
& 0.40 &0 \\
fidap031 & $\dagger$&$\dagger$ &&1.66 &18.70 & 137&$\dagger$ & 0.65 &2&&1.68 &20.71
& 141&801 & 0.40 &0 \\
fidap036 & $\dagger$&$\dagger$ &&1.76 &5.99 & 85&250 & 0.40 &0&&1.76 &5.99 & 85&250
& 0.40 &0 \\
nos3 & 213& $\dagger$ &&0.50 &0.40 & 106&536 & 0.38 &0&&0.50 &0.40 & 106&536 & 0.38 &0 \\
nos6 & $\dagger$&$\dagger$ &&0.56 &0.14 & 38&44 & 0.40 &0&&0.56 &0.14 & 38&44 & 0.40 &0 \\
orsirr\_1 & $\dagger$&$\dagger$ &&3.19 &1.79 & 37&59 & 0.39 &0&&3.19 &1.79 & 37&59
& 0.39 &0 \\
orsirr\_2 & $\dagger$&$\dagger$ &&3.26 &1.49 & 38&60 & 0.39 &0&&3.26 &1.49 & 38&60
& 0.39 &0 \\
orsreg\_1 & 687&346 &&2.13 &4.02 & 40&67 & 0.38 &0&&2.13 &4.02 & 40&67 & 0.38 &0 \\
pores\_2 & $\dagger$&$\dagger$&&3.53 &1.97&53&146&0.68&3&&3.58&2.04&59&147&0.40&0  \\
sherman1 & 356&$\dagger$&&1.62 &0.49 & 37&60 & 0.43 &1&&1.63 &0.49 & 36&60 & 0.40 &0 \\
sherman2 &$\dagger$&$\dagger$&&2.42 &3.59 & 5&8 & 0.40 &0&&2.42 &3.59 & 5&8 & 0.40 &0 \\
sherman3 & $\dagger$&$\dagger$&&1.15 &3.33 & 201&$\dagger$ & 0.40 &0&&1.15 &3.33 & 201
&$\dagger$ & 0.40 &0 \\
sherman4 & 101&377 &&0.88 &0.25 & 41&59 & 0.40 &0&&0.88 &0.25 & 41&59 & 0.40 &0 \\
sherman5 & $\dagger$&$\dagger$ &&1.18 &1.64 & 35&53 & 0.40 &0&&1.18 &1.64 & 35&53 & 0.40 &0 \\
\noalign{\smallskip}\hline
\end{tabular}
\end{table}
\par

Now we take a closer look at PSAI($tol$). The table shows that for
$\varepsilon=0.2, 0.3, 0.4$, Algorithm~\ref{alg2} used $l_{\max}=11,10,8$ to
attain the accuracy requirements, respectively. If we reduced $l_{\max}$ to
$8, 6, 5$,
there are only very few columns of $M$ for only a few matrices which
do not satisfy the accuracy requirements, but the corresponding $r_{\max}$
are still reasonably small and exceed $\varepsilon$ no more than twice. This
indicates that the corresponding $M$ are still effective preconditioners, as
confirmed by the iterations used, but they are generally less effective than
the corresponding ones obtained by the bigger $l_{\max}$ which guarantee that
$M$ computed by PSAI($tol$) succeeds for very small $l_{\max}$.
Table~\ref{table2} clearly tells us that for a smaller
$\varepsilon$, PSAI($tol$) needs larger $l_{\max}$ for the while loop. But a remarkable
finding is that PSAI($tol$) succeeds for very small $l_{\max}$. Given a rather
mildly small $\varepsilon$ like 0.3 and the generality of test problems,
these experiments suggests that we may well set $l_{\max}=10$ as a default value
in Algorithm~\ref{alg2}.

We observe from Table~\ref{table2} that for each problem the smaller $\varepsilon$
the fewer iterations the two Krylov solvers use. However, in the experiments,
we notice that, for all problems except fidap024, fidap31 and sherman3,
for which GMRES(50) failed when $\varepsilon=0.4$,
and given $\varepsilon$ and $l_{\max}$, setup time of
ptime of M and Krylov iterations only occupy a very small percent.
As we have addressed in the introduction,
this is a typical feature of an effective SAI preconditioning procedure and
has been recognized widely in the
literature, e.g., \cite{benzi1998bit}. This is true even in a parallel
computing environment. Moreover, our Matlab codes have not been optimized, and thus
may give rise to lower performance. Thus,
we do not list the time for the Krylov iterations in Table~\ref{table2}.
With this in mind, we find from Table~\ref{table2} that for
the first five matrices, orsreg\_1 and pores\_2, the sparsity and construction cost of $M$
 increases considerably as $\varepsilon$ decreases. Overall, to tradeoff
 effectiveness and general application, $\varepsilon=0.3$ is a good choice for accuracy
 and the maximum number of while loops in PSAI($tol$) should be
$10$.

Regarding Table~\ref{table2}, we finally point out a very important fact:
for each of the test problems and given three choices for $\varepsilon$, BPSAI
and PSAI($tol$) with our dropping criterion use exactly the same value for $l_{\max}$
to yield preconditioners attaining accuracy $\varepsilon$.
This fact is important because it illustrates that the latter
behaves like the former with the same choice for $l_{\max}$, while obtaining an equally
effective preconditioner at less computational cost for setup.

The next results illustrate three considerations.
First, choosing a smaller $tol_k$ is not required because the resulting $M_d$
is more dense and costs more to set up but
is not necessarily a better preconditioner. Second, for an improperly chosen fixed small $tol$,
that is, $tol>\frac{\varepsilon}{nnz(\tf{m}_k)\|A\|_1}$
at some while-loops of Algorithm~\ref{alg2}, PSAI($tol$)
may produce a numerically singular $M_d$ which will cause the complete failure
of the preconditioning. Third, for a $tol$ that produces a singular $M_d$,
reducing $tol$ by one order of magnitude, will yield an $M_d$
which is a good preconditioner but is less effective than the $M_d$
obtained with $tol_k$ defined by (\ref{muepsilon2}).
This illustrates that choosing a fixed $tol$ empirically is at risk
for generating an ineffective $M_d$.

To illustrate the first consideration, we use the three matrices
orsirr\_1, orsirr\_2 and orsreg\_1 and use PSAI($tol$) with $\varepsilon=0.2$, $l_{\max}=8$
and with $tol_k$ ranging from a little
smaller to considerably smaller than that indicated by (\ref{muepsilon2}).
Specifically, denote the right hand side in (\ref{muepsilon2}) by RHS, then we use
$RHS$, $RHS/2$, $RHS/10$ and $RHS/100$, and
investigate the impact of the choice for the tolerance on the quality, sparsity
and computational cost of setup of $M_d$.
We report the results in Table~\ref{smallmu},
where the tolerance $tol_k=0$ corresponds to the BPSAI procedure. For the three matrices,
as $coln$ in Table~\ref{table2} and $r_{\max}$ in Table~\ref{smallmu} indicate, the
approximate inverses $M$ obtained by PSAI($tol$)
with these different tolerances $tol_k$ and BPSAI
have attained the accuracy $\varepsilon$. For each of these three problems,
we can easily observe that $M$
becomes increasingly denser as $tol_k$ decreases and $M$ is the densest for $tol_k=0$.
However, the preconditioning quality of denser $M$ is not improved, since the
corresponding numbers of
Krylov iterations are almost the same, as shown by $iter\_\,b$ and
$iter\_\,g$. Moreover, we can see that the setup time $ptime$ of $M$
increases as $tol_k$ decreases. For all the other test problems in Table~\ref{table1},
we have also made numerical experiments in the above way. We find that the sparsity and
preconditioning quality of $M$ obtained by PSAI($tol$) with the five $tol_k$
changes very little. This means that our dropping criterion (\ref{muepsilon2}) enables
us to drop entries of small magnitude in $M$ and smaller $tol_k$ does not help any.
Together with Table~\ref{smallmu}, we conclude that our dropping
criterion is effective and robust and it is not necessary to take smaller
$tol_k$ in PSAI($tol$).

\begin{table}
\begin{center}
 \caption{Effects of smaller $tol_k$ for PSAI($tol$) with
 $\varepsilon=0.2$ and $l_{\max}=8.$
  Note: when the iterations for BiCGStab are $k$, the dimension of
the Krylov subspace is $2k$.}\label{smallmu}
\begin{tabular}{ccccccc}\hline\noalign{\smallskip}
&&$tol_k=RHS$&$tol_k=\frac{RHS}{2}$&$tol_k=\frac{RHS}{10}$
&$tol_k=\frac{RHS}{100}$&$tol_k=0$\\
\noalign{\smallskip}\hline\noalign{\smallskip}
\multirow{4}*{orsirr\_1}& $spar$ & 10.15 & 10.81 & 12.02 & 13.13 & 16.77\\
& $ptime$ & 7.41 & 7.60 & 8.34 & 8.56 & 9.05 \\
& $iter\_\,b,iter\_\,g$ & 15, ~26 & 15, ~26 &15,~26 & 15,~26 &15, ~26\\
& $r_{\max}$ & 0.199974 & 0.199971 &0.199970 &0.199970&0.199970 \\
\noalign{\smallskip}\noalign{\smallskip}
\multirow{4}*{orsirr\_2}& $spar$ & 10.71 & 11.29 & 12.42 & 13.52 &16.70\\
& $ptime$ & 6.53 & 6.58 & 6.68 & 7.45 & 8.11 \\
& $iter\_\,b,iter\_\,g$ & 16,~25& 16,~25 &14,~25 & 14,~25&14,~25 \\
& $r_{\max}$ &0.199974 & 0.199970 &0.199970& 0.199970&0.199970 \\
\noalign{\smallskip}\noalign{\smallskip}
\multirow{4}*{orsreg\_1}& $spar$ & 9.16 & 9.63 & 11.27 & 12.97 &16.82\\
& $ptime$ & 19.49 & 20.42 & 23.35 & 25.84&35.47 \\
& $iter\_\,b,iter\_\,g$ & 18,~29 &18,~29 &18,~29 & 18,~29 &18,~29\\
& $r_{\max}$ &0.199853 & 0.199841  &0.199840 & 0.199839&0.199839 \\
\noalign{\smallskip}\hline
\end{tabular}
\end{center}
\end{table}

\begin{table}[!htb]
\begin{center}
\caption{Sensitivity of the quality of $M$ to fixed dropping tolerance $tol$}
\label{tolsensitive}
\vskip 0.6cm
 \begin{tabular}{ccccccc}
\hline\noalign{\smallskip}
 &nos3&nos6&orsirr\_1&orsirr\_2&orsreg\_1&sherman5\\
 \noalign{\smallskip}\hline\noalign{\smallskip}
$tol$&$10^{-2}$&$10^{-6}$&$10^{-3}$&$10^{-3}$&$10^{-3}$&$10^{-2}$\\
$r_{\max}$&3.00&1.93&285.17&71.34&23.12&24.71\\
$mintol$&$1.61\times10^{-6}$&$3.97\times10^{-10}$&$8.48\times10^{-10}$&
$8.48\times10^{-10}$&$1.44\times10^{-8}$&$2.10\times10^{-7}$\\
$maxtol$&$4.34\times10^{-5}$&$6.25\times10^{-9}$&$8.80\times10^{-8}$&
$1.17\times10^{-7}$&$1.39\times10^{-6}$ &$7.91\times10^{-6}$\\
\noalign{\smallskip}\hline
 \end{tabular}\vskip 0.2cm
\centerline{(a):\quad Bad $tol$ resulting in numerically singular $M$
for $\varepsilon=0.2, l_{\max}=8$ }
\vskip 0.4cm
 \begin{tabular}{ccccccc}
\hline\noalign{\smallskip}
Matrix&$tol$&$ptime$&$spar$&$iter\_\,b$&$iter\_\,g$&$r_{\max}$\\
\noalign{\smallskip}\hline\noalign{\smallskip}
nos3& $10^{-3}$& 3.53 & 0.67 & 162& $\dagger$&0.68\\
nos6& $10^{-7}$& 0.47 & 1.78 & 79&72&0.73  \\
orsirr\_1& $10^{-4}$&7.09 & 2.31 & 21& 34&14.11  \\
orsirr\_2& $10^{-4}$&6.32 & 2.62 & 41& 50&14.11 \\
orsreg\_1& $10^{-4}$&17.86 & 3.05 & 25& 39&2.33\\
sherman5 & $10^{-3}$&4.64 & 1.72 & 22& 32&4.14 \\
\noalign{\smallskip}\hline
 \end{tabular}\vskip 0.2cm
\centerline{(b):\quad Good $tol$ leading to effective $M$ for
$\varepsilon=0.2, l_{\max}=8$.}
\end{center}
\end{table}

\begin{figure}\label{fig1}
\begin{center}
\includegraphics[width=14cm]{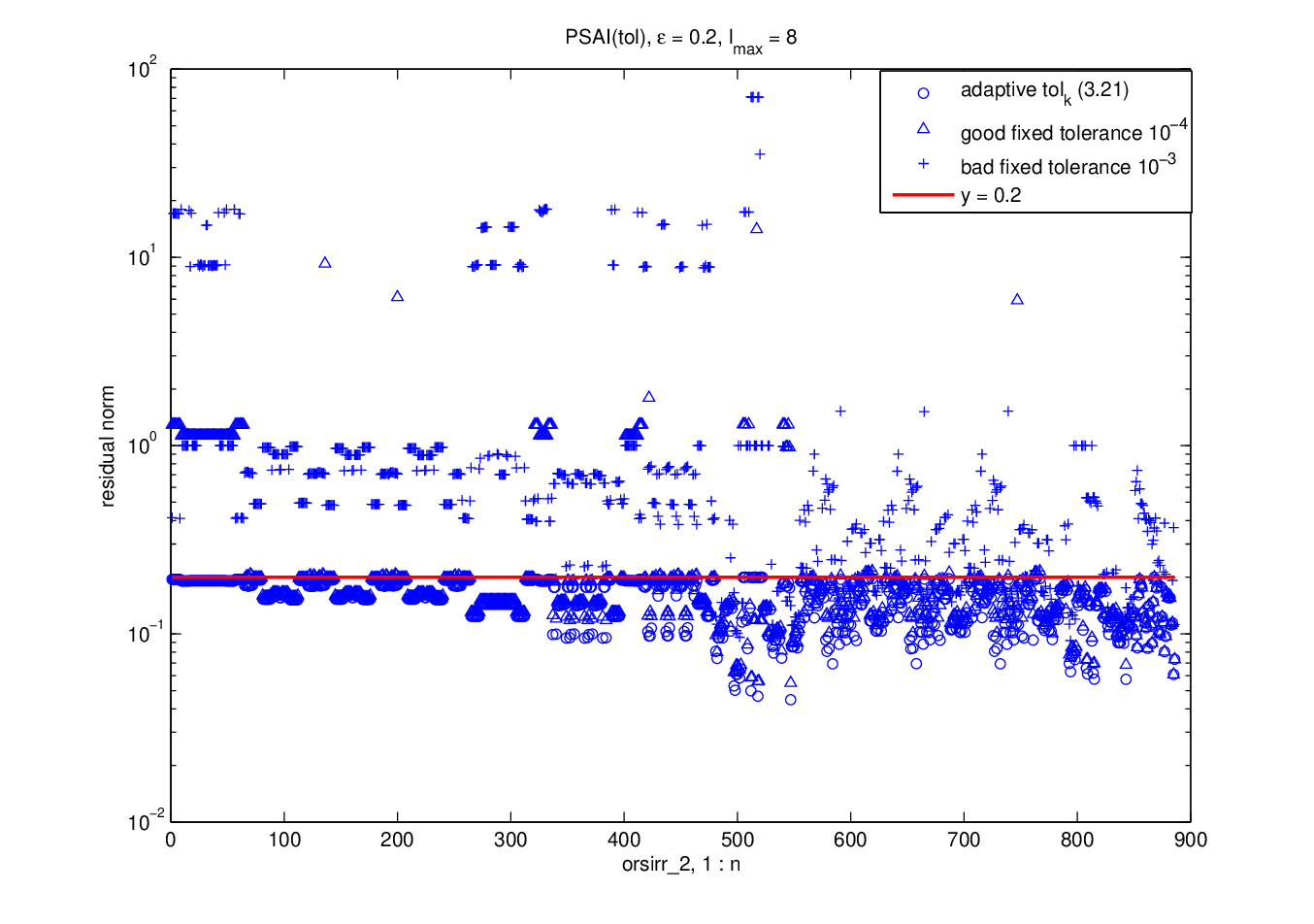}
\caption{Column residual norms of $M$ for orsirr\_2 obtained by PSAI($tol$)
with bad and good fixed $tol$ and adaptive $tol_k$ defined by (\ref{muepsilon2})}
\end{center}
\end{figure}
\par

To illustrate the second and third consideration, we investigate the behavior of $M$
obtained by PSAI($tol$) for improperly chosen dropping tolerance $tol$ that
seems small intuitively. We attempt to show that a choice of fixed $tol$ that is apparently
small, but bigger than that defined by (\ref{muepsilon2}) for some $k$ may produce
a numerically singular $M$. Specifically, we take
$$
tol>\frac{\varepsilon}
{nnz(\tf{m}_k)\|A\|_1},
$$
in the while-loop of Algorithm~\ref{alg2}, where the right-hand side
is just our dropping tolerance (\ref{muepsilon2}). We drop the entries whose sizes
are below such improper $tol$. Table~\ref{tolsensitive}(a) lists the matrices,
each with the dropping tolerance $tol$ that leads to a numerically singular $M$
for $\varepsilon=0.2,\ l_{\max}=8$. The $mintol$ and $maxtol$ in
Table~\ref{tolsensitive}(a) denote the minimum and maximum
of $tol_k$ defined by (\ref{muepsilon2}). However, if we decrease the tolerance
$tol$ by one order of magnitude, we will obtain good preconditioners; see
Table~\ref{tolsensitive}(b) for details.
We emphasize that for the given $\varepsilon$ and $l_{\max}$ and all the matrices
in Table~\ref{tolsensitive}(b), PSAI($tol$) with dropping criterion (\ref{muepsilon2})
has computed the sparse approximations $M$ of $A^{-1}$ with the desired
accuracy $\varepsilon$, as shown in Table~\ref{table2}.

We see from Table~\ref{tolsensitive} (a) that the maximum residual
$r_{\max}$ for each problem is not small at all for the chosen bad fixed dropping
tolerance $tol$. On the other hand, Table~\ref{tolsensitive} (b) indicates
that the one order reduction of $tol$ results in essential improvements on the effectiveness
of preconditioners, not only delivering nonsingular $M$ but also accelerating
the convergence considerably. These tests indicate that the non-singularity
and quality of $M_d$ obtained by PSAI($tol$) can be very sensitive to the choice
of dropping tolerance $tol$. However,
compared with the corresponding results for $\varepsilon=0.2,l_{\max}=8$
on the same test problems in Table~\ref{tolsensitive} (b) and Table \ref{table2},
we find that the preconditioner obtained by PSAI($tol$) with the good fixed tolerance $tol$
is not so effective as that with $tol_k$ defined by (\ref{muepsilon2}),
as shown by values of $iter\_b$ and $iter\_g$. Indeed, the preconditioners obtained
by fixed tolerance $tol$ do not satisfy
the accuracy $\varepsilon$, as $r_{\max}$ indicate.

To be more illustrative, for orsirr\_2 we depict the residual norms
$\|A\tf{m}_k-\tf{e}_k\|, k=1,2, \ldots,n$ of three such $M$ obtained by PSAI($tol$)
with the adaptive $tol_k$ defined by (\ref{muepsilon2}) and bad to good fixed
$tol=10^{-3},\  10^{-4}$; see Figure~\ref{fig1}, where
the solid line $y=\varepsilon=0.2$ parallel to the $x$-axis denotes
our accuracy requirement, the circle `$\circ$', the plus `$+$' and
the triangle `$\vartriangle$' are $\|A\tf{m}_k-\tf{e}_k\|, k=1,2, \ldots,n$
of each $M$. We find from the figure that all the circles `$\circ$ fall
below the solid line,
meaning that PSAI($tol$) with $tol_k$ defined by (\ref{muepsilon2}) computes
all the columns of $M$ with desired accuracy; many `$+$' reside
above the solid line and some of them are far away from $\varepsilon=0.2$
and can be up to $10\sim 100$,
indicating that $M$ obtained by PSAI($tol$) is very bad and
of poor quality for preconditioning; most of the triangles `$\vartriangle$'
are below $\varepsilon=0.2$, and a small part of them is above it, revealing that
$M$ is improved very substantially but is not so good like $M$ computed
by PSAI($tol$) with $tol_k$ defined by (\ref{muepsilon2}).

Table~\ref{tolsensitive} and Figure~\ref{fig1}
tell us that empirically chosen tolerances
are problematic and susceptible to failure.
In contrast, Tables~\ref{table2}--\ref{tolsensitive} demonstrate
that our selection criterion (\ref{muepsilon2}) is very robust
for PSAI($tol$).

\subsection{Results for three static SAI procedures}

As an application of our theory, in this subsection,
we test the static F-norm minimization based SAI
preconditioning procedures with the three popular patterns of $(I+A)^3$,
$(I+|A|+|A^T|)^3A^T$ and $(AA^T)^2A^T$, respectively;
see \cite{huckle99} for the effectiveness of these patterns. We
attempt to show the effectiveness of dropping criterion
(\ref{staticcriterion}) and exhibit the sensitiveness of the
preconditioning quality of $M$ to dropping tolerances $tol_k$.
We first compute $M$ by predetermining its pattern and
solving $n$ independent LS problems,
and then get a sparser $M_d$ by dropping the entries of small magnitude in $M$
below the tolerance defined by (\ref{staticcriterion}) or some empirically chosen
ones.

\begin{table}[ht]
\begin{center}
\caption{Sensitivity of the quality of $M_d$ to some fixed
$tol$ for the static SAI procedure with the pattern of $(I+A)^3$.
Note: when the iterations for BiCGStab are $k$, the dimension of the Krylov
subspace is $2k$.
}
\label{saisensitive}
\vskip 0.6cm
 \begin{tabular}{cccccc}
\hline\noalign{\smallskip}
 &orsirr\_1&orsirr\_2&orsreg\_1&pores\_2&sherman5\\ \noalign{\smallskip}
 \hline\noalign{\smallskip}
$tol$&$10^{-5}$&$10^{-5}$&$10^{-3}$&$10^{-6}$&$10^{-2}$\\
$r_{\max}$&1.32&1.00&15.2&18.0&24.7\\
$mintol$&$1.37\times10^{-9}$&$1.37\times10^{-9}$&$4.93\times10^{-8}$&
$6.53\times10^{-12}$&$1.63\times10^{-7}$\\
$maxtol$&$2.64\times10^{-8}$&$2.23\times10^{-8}$&$3.52\times10^{-7}$&
$1.55\times10^{-10}$&$2.37\times10^{-5}$\\
\noalign{\smallskip}\hline
 \end{tabular}\vskip 0.2cm
\centerline{(a):\quad Bad $tol$ resulting in numerically singular $M_d$ }
\vskip 0.4cm
 \begin{tabular}{ccccccc}
\hline\noalign{\smallskip}
Matrix&$tol$&$ptime$&$spar$&$iter\_\,b$&$iter\_\,g$&$r_{\max}$\\
\noalign{\smallskip}\hline\noalign{\smallskip}
orsirr\_1& $10^{-6}$&1.63 & 1.82 & 33& 50&0.42  \\
orsirr\_2& $10^{-6}$&1.31 & 2.69 & 32& 48&0.42 \\
orsreg\_1& $10^{-4}$&4.09 & 0.91 & 45& 74&1.32\\
pores\_2& $10^{-7}$ &3.10& 2.47&124&158&1.74\\
sherman5 & $10^{-3}$&11.79 & 1.55 & 24& 34&3.76 \\
 \noalign{\smallskip}\hline
 \end{tabular}\vskip 0.2cm
\centerline{(b):\quad Good $tol$ leading to effective $M_d$}
\end{center}
\end{table}
\par

We summarize the results in Tables~\ref{saisensitive}--\ref{static3},
where $ptime$ includes the time for predetermination of the pattern of $M$,
the computation of $M$ and the sparsification of $M$,
and $stime\_b$ and $stime\_g$ denote the CPU time in second of
BiCGStab and GMRES(50) applied to solve the preconditioned
linear systems. We observed that there are some columns whose residual norms
$\|A\tf{m}_k-\tf{e}_k\|=\varepsilon_k$ are very small (some are
at the level of $\epsilon_{\rm mach}$). Therefore, to drop entries of small
magnitude as many as possible, we replace those $\varepsilon_k$
below $0.1$ by $0.1$ in (\ref{staticcriterion}).

We test the static SAI procedure with the pattern of $(I+A)^3$.
Table~\ref{saisensitive}(a) lists the matrices, each with the fixed tolerance
$tol$ leading to a numerically singular $M$ and Table \ref{saisensitive}(b)
exhibits the good performance of $M_d$ generated from the static SAI
by decreasing the corresponding $tol$ in Table~\ref{saisensitive} (a)
by one order of magnitude.  Tables~\ref{static1}--\ref{static3} show the
results obtained by the three static SAI
procedures with dropping criterion (\ref{staticcriterion}).

\begin{table}[ht]
\begin{center}
\caption{Static SAI procedure with the pattern of $(I+A)^3$.
Note: when the iterations for BiCGStab are $k$, the dimension of the Krylov
subspace is $2k$.
}\label{static1}
 \begin{tabular}{ccccccccc}
\hline\noalign{\smallskip}
\multicolumn{2}{c}{}&$ptime$&$spar$&$iter\_\,b$&$iter\_\,g$&$stime\_\,b$&$stime\_\,g$&$r_{\max}$\\
\noalign{\smallskip}\hline\noalign{\smallskip}
\multirow{2}*{orsirr\_1}& $M$ & 1.65 & 8.36 & 29& 45& 0.03& 0.08&0.42 \\
& $M_d$ &1.78& 4.54 & 29& 45 & 0.01& 0.06&0.42\\ \noalign{\smallskip}\noalign{\smallskip}
\multirow{2}*{orsirr\_2}& $M$ & 1.48 & 8.62 & 30& 44 & 0.02& 0.04&0.42 \\
& $M_d$ &1.56& 5.24 & 30& 44 & 0.02& 0.03&0.42\\ \noalign{\smallskip}\noalign{\smallskip}
\multirow{2}*{orsreg\_1}& $M$ & 4.39 & 7.53 & 28& 51& 0.04& 0.09&0.42 \\
& $M_d$ &4.69& 2.95 & 33& 51 & 0.01& 0.08&0.42 \\ \noalign{\smallskip}\noalign{\smallskip}
\multirow{2}*{pores\_2}& $M$ & 2.74 & 9.25 & 52& 118 & 0.09& 0.14&0.94\\
& $M_d$ &3.00& 4.98 &52& 118&0.06& 0.13&0.94\\ \noalign{\smallskip}\noalign{\smallskip}
\multirow{2}*{sherman5}& $M$ & 13.26 & 8.39 & 22& 31 & 0.04& 0.05&0.32 \\
& $M_d$ &14.07& 3.54 & 22& 31 & 0.02& 0.04&0.32 \\
 \noalign{\smallskip}\hline
 \end{tabular}
\end{center}
\end{table}

\begin{figure}\label{fig2}
 \begin{center}
  \includegraphics[width=14cm]{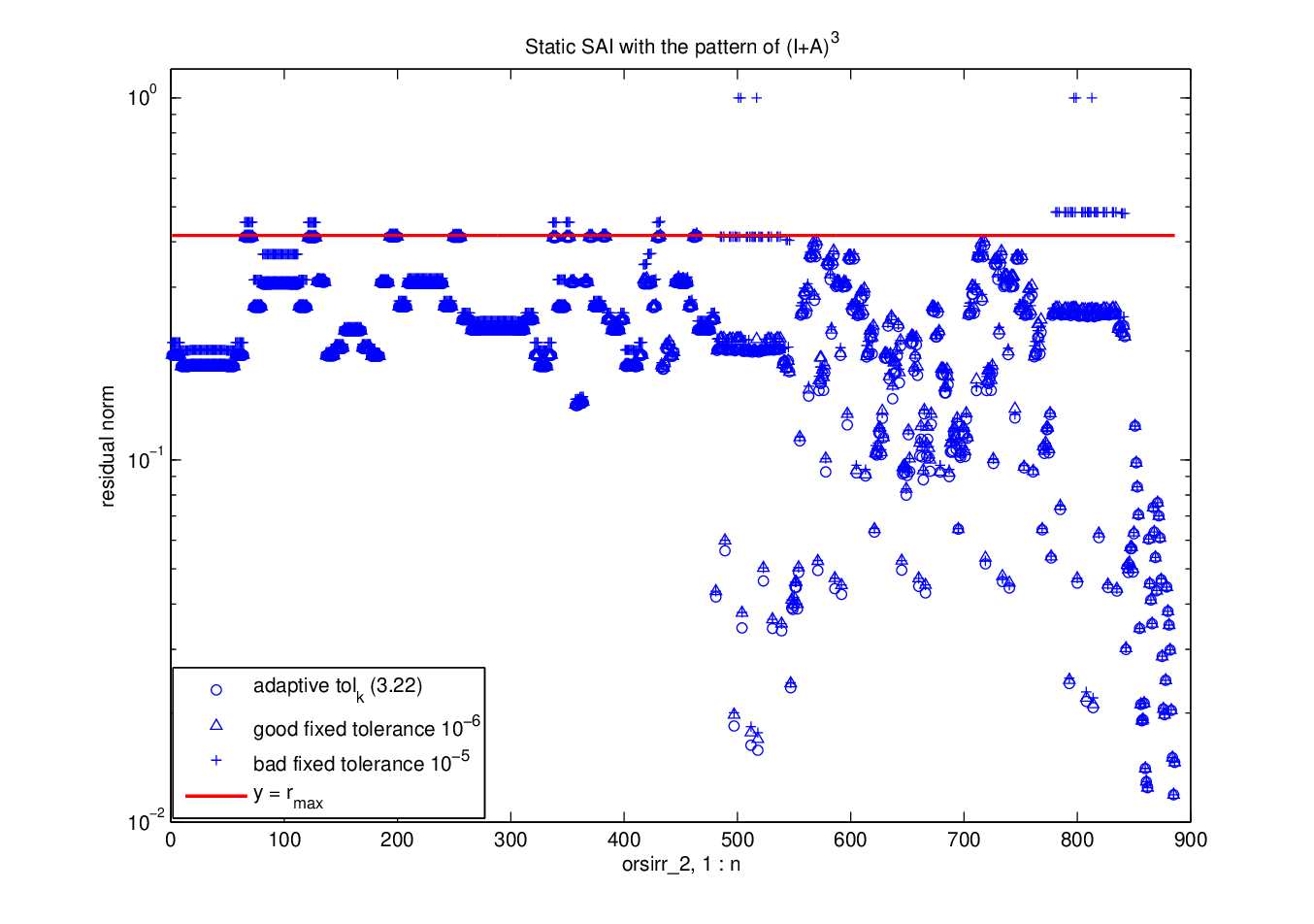}
  \caption{Column residual norms of $M_d$ for orsirr\_2 obtained by the static SAI
with bad and good fixed $tol$ and $tol_k$ defined by (\ref{staticcriterion})}
 \end{center}
\end{figure}
\begin{table}[ht]
\begin{center}
\caption{Static SAI procedure with the pattern of $(I+|A|+|A^T|)^3A^T$.
Note: when the iterations for BiCGStab are $k$, the dimension of the Krylov subspace
is $2k$.
}
\label{static2}
 \begin{tabular}{ccccccccc}
\hline\noalign{\smallskip}
\multicolumn{2}{c}{}&$ptime$&$spar$&$iter\_\,b$&$iter\_\,g$&$stime\_\,b$&$stime\_\,g$&$r_{\max}$\\
\noalign{\smallskip}\hline\noalign{\smallskip}
\multirow{2}*{orsirr\_1}& $M$ & 5.45 & 16.41 & 18& 28 & 0.03& 0.04&0.32\\
& $M_d$ &6.00& 10.06 & 18& 28 & 0.01& 0.03&0.32 \\ \noalign{\smallskip}\noalign{\smallskip}
\multirow{2}*{orsirr\_2}& $M$ & 4.64 & 17.03 & 18& 28 & 0.02& 0.05&0.32 \\
& $M_d$ &5.19& 11.23 & 16& 28 & 0.01& 0.02&0.32\\ \noalign{\smallskip}\noalign{\smallskip}
\multirow{2}*{orsreg\_1}& $M$ & 14.86 & 14.03 & 19& 34& 0.05& 0.06&0.34\\
& $M_d$ &16.82& 6.83 & 19& 34 & 0.02& 0.05&0.34\\ \noalign{\smallskip}\noalign{\smallskip}
\multirow{2}*{pores\_2}& $M$ & 11.02 & 18.42 & 26& 38& 0.05& 0.07&0.86\\
& $M_d$ &13.25& 11.91 &26& 38&0.05& 0.06&0.86\\ \noalign{\smallskip}\noalign{\smallskip}
\multirow{2}*{sherman5}& $M$ & 52.36 & 14.94 & 16& 23& 0.06& 0.06&0.25 \\
& $M_d$ &56.50& 6.41 & 16& 23 & 0.02& 0.04&0.25\\
\noalign{\smallskip}\hline
 \end{tabular}
\end{center}
\end{table}
\begin{table}[htb]
\begin{center}
\caption{Static SAI procedure with the pattern of $(A^TA)^2A^T$.
Note: when the iterations for BiCGStab are $k$, the dimension of the
Krylov subspace is $2k$.
}
\label{static3}
 \begin{tabular}{ccccccccc}
\hline\noalign{\smallskip}
\multicolumn{2}{c}{}&$ptime$&$spar$&$iter\_\,b$&$iter\_\,g$&$stime\_\,b$&$stime\_\,g$&$r_{\max}$\\
\noalign{\smallskip}\hline\noalign{\smallskip}
\multirow{2}*{orsirr\_1}& $M$ & 15.37 & 27.79 & 13& 20 & 0.02& 0.02&0.24\\
& $M_d$ &18.26& 18.39 & 13& 20 & 0.01& 0.02&0.24 \\ \noalign{\smallskip}\noalign{\smallskip}
\multirow{2}*{orsirr\_2}& $M$ & 12.11 & 28.84 & 14& 19& 0.02& 0.03&0.24\\
& $M_d$ &14.20 &20.26& 14& 19 & 0.02& 0.02&0.24\\ \noalign{\smallskip}\noalign{\smallskip}
\multirow{2}*{orsreg\_1}& $M$ &49.42 & 22.77 & 14& 24 & 0.05& 0.04&0.31\\
& $M_d$ &57.22& 12.75 & 14& 24 & 0.02& 0.04&0.31 \\ \noalign{\smallskip}\noalign{\smallskip}
\multirow{2}*{pores\_2}& $M$ & 41.83 & 30.84 & 16& 26& 0.06& 0.05&0.68\\
& $M_d$ &49.71& 19.46 &16& 26&0.05& 0.04&0.68\\ \noalign{\smallskip}\noalign{\smallskip}
\multirow{2}*{sherman5}& $M$ & 129.25 & 22.53 & 14& 19& 0.06& 0.06&0.20 \\
& $M_d$ &138.41& 9.09 & 14& 19 & 0.02& 0.03&0.20 \\
 \noalign{\smallskip}\hline
 \end{tabular}
\end{center}
\end{table}
\par

Singular $M_d$ as in Table~\ref{saisensitive} (a) lead to complete failure of
preconditioning. We also see from the table that all the maximum residuals
$r_{\max}$ of $M_d$ for the five matrices are not small, meaning that the
$M_d$ are definitely ineffective for preconditioning. But Table~\ref{saisensitive} (b)
shows that the situation is improved drastically when the corresponding tolerances
$tol$ are decreased only by one order of magnitude.

Similar to PSAI($tol$), the quality of static SAI
preconditioners depends on, and can be very sensitive to, the dropping tolerances.
Figure~\ref{fig2} depicts the residual norms of three $M_d$ obtained
by the static SAI procedure with the pattern of $(I+A)^3$ using the bad $tol=10^{-5}$,
the good $tol=10^{-6}$ and our criterion (\ref{staticcriterion}), which are
denoted by the plus `$+$', the triangle `$\vartriangle$ and circle
`$\circ$', respectively, and the solid line $y=r_{\max}$ parallel
to the $x$-axis is the maximum column residual norm of $M_d$ obtained with
 (\ref{staticcriterion}). We see from the figure
that $M_d$ constructed with (\ref{staticcriterion}) and the good
fixed tolerance $tol=10^{-6}$ are fairly good but the former one is more effective
than the latter one, since the triangles `$\vartriangle$' are either
indistinguishable with or a little bit higher than the corresponding
circles `$\circ$'. Such effectiveness is also reflected in the values of
$iter\_b$ and $iter\_g$ in Table~\ref{saisensitive} and
Table~\ref{static1}. In contrast, $M_d$ obtained with the tolerance $tol=10^{-5}$ has
many columns, which are poorer than those of $M_d$ obtained with (\ref{staticcriterion}),
since the `$+$' are above the corresponding `$\circ$', and it
has some columns whose residual norms reside above the solid line $y=r_{\max}$.

For Tables~\ref{static1}--\ref{static3}, we see that each $M_d$ is
sparser than the corresponding $M$ and it is cheaper to apply $M_d$ than $M$
in Krylov solvers, as $stime\_b$ and $stime\_g$ indicate.
Furthermore, for each matrix, since we use (\ref{staticcriterion}) to only
drop the entries of small magnitude, two $r_{\max}$ corresponding
to each pair $M$ and $M_d$ are approximately the same and they are fairly small.
So it is expected that
each $M_d$ and the corresponding $M$ have very similar accelerating quality.
This is indeed the case, because for all the problems but orsreg\_2, each Krylov
solver preconditioned by $M_d$ and the corresponding $M$ uses exactly the the same
number of iterations to achieve convergence. For orsreg\_2 in Table~\ref{static1},
BiCGSTab preconditioned by $M_d$ uses only three more iterations
than it preconditioned by $M$. These results demonstrate
that our selection criterion (\ref{staticcriterion}) is
effective and robust. Compared with Table~\ref{saisensitive},
we see from Table~\ref{static1} that the SAI preconditioning with our criterion
(\ref{staticcriterion}) is more effective than that with
the good fixed tolerance $tol=10^{-6}$, since the maximum residual norms $r_{\max}$ for the
former are always not bigger than those for the latter and
the Krylov solvers preconditioned by the former used fewer iterations to achieve convergence.
In addition, we notice from Table~\ref{static3}
that the pattern of $(A^TA)^2A^T$ leads to considerably denser $M$
and $M_d$ that are good approximate inverses but are much more
expensive to compute, compared with the other two patterns.
Therefore, as far as the overall performance is concerned, this static
SAI procedure is less effective than the other two.

\section{Conclusions}

Selection criteria for dropping tolerances are vital to SAI preconditioning.
However, this important problem has received little attention and
never been studied rigorously and systematically. For F-norm minimization
based SAI preconditioning, such criteria
affect the non-singularity, the quality and effectiveness of a preconditioner $M$.
An improper choice of dropping tolerance
may produce a numerically singular $M$,
causing the complete failure in preconditioning, or may produce
a good but denser $M$ possibly at more cost for setup and application.
To develop a robust PSAI($tol$) preconditioning procedure,
we have analyzed the effects of dropping tolerances on
the non-singularity, quality and effectiveness of preconditioners.
We have established some important and intimate relationships
between them. Based on them, we have proposed adaptive robust
selection criteria for dropping tolerances that can make $M$
as sparse as possible and of comparable quality to those obtained by BPSAI, so
that it is possible to lower the cost of setup and application.
The theory on selection criteria has been adapted to static F-norm minimization
based SAI preconditioning procedures. Numerical experiments have shown that
our criteria work very well. However, we point out that it is more
important and beneficial to perform dropping in the adaptive PSAI preconditioning
procedure than a static SAI one.

For general purposes and effectiveness, robust selection criteria for dropping
tolerances also play a key role in other {\em adaptive} F-norm minimization
based SAI preconditioning procedures whenever dropping is used.
Just like for PSAI($tol$), dropping criteria serve two
purposes, one of which is to make an approximate inverse $M$
as sparse as possible and the other is to guarantee its comparable
preconditioning quality to that obtained from
SAI procedure without dropping. For adaptive factorized sparse
approximate inverse preconditioning, such as AINV type algorithms
\cite{benzi2002preconditioning,benzi1998bit},
dropping is equally important. Different from F-norm minimization based
SAI preconditioning, the non-singularity
of the factorized $M$ is guaranteed naturally. Nonetheless,
how to drop entries of small magnitude is nontrivial and has not yet been well studied.
All of these are significant and are topics for further consideration.

\begin{acknowledgements}
We thank the anonymous referees very much for their valuable suggestions and
comments that helped us improve presentation of the paper substantially.
\end{acknowledgements}


\end{document}